\documentclass{amsart}[12]

\usepackage{amsmath, amssymb,epsfig, verbatim}
\usepackage{graphicx}
\usepackage{epic}
\usepackage{color}
\usepackage{enumerate, mathabx}
\usepackage{hyperref}

\newtheorem{thm}{Theorem}[section]

\newtheorem{cor}[thm]{Corollary}

\numberwithin{equation}{section}

\newcommand{\x}{\times}
\newcommand{\s}{\mathfrak {s}}

\newcommand{\C}{\mathbb C}
\newcommand{\Z}{\mathbb Z}
\newcommand{\N}{\mathbb N}
\newcommand{\Q}{\mathbb Q}
\newcommand{\R}{\mathbb R}

\newcommand{\RP}{{\mathbb R}{\mathbb P}}
\newcommand{\cpkk}{{\overline {{\mathbb C}{\mathbb P}^2}}}
\newcommand{\cpk}{{\mathbb {CP}}^2}

\DeclareMathOperator{\SW}{SW}

\newtheorem{corollary}[thm]{Corollary}
\newtheorem{lemma}[thm]{Lemma}
\newtheorem{proposition}[thm]{Proposition}
\newtheorem{question}[thm]{Question}

\theoremstyle{definition}
\newtheorem*{definition*}{Definition}
\newtheorem{definition}[thm]{Definition}
\newtheorem{remark}[thm]{Remark}

\newtheorem{example}[thm]{Example}

\newcommand{\blue}{\textcolor{blue}} 

\newcommand{\twprod}{\mathbin{\mathchoice%
    {\ooalign{\raise1.15ex\hbox{$\scriptstyle\sim$}\cr\hidewidth$\times$\hidewidth\cr}}%
    {\ooalign{\raise1.15ex\hbox{$\scriptstyle\sim$}\cr\hidewidth$\times$\hidewidth\cr}}%
    {\ooalign{\raise.85ex\hbox{$\scriptscriptstyle\sim$}\cr\hidewidth$\scriptstyle\times$\hidewidth\cr}}%
    {\ooalign{\raise.65ex\hbox{$\scriptscriptstyle\sim$}\cr\hidewidth$\scriptscriptstyle\times$\hidewidth\cr}}%
    }}

\newcommand{\CPb}{\overline{\mathbb{CP}}{}^{2}}
\newcommand{\CP}{{\mathbb{CP}}{}^{2}}

\newcommand{\Diff}{\operatorname{Diff{^+}}}

\begin{document}

\title[Four-manifolds with finite cyclic fundamental groups]{Smooth structures on four-manifolds \\ with finite cyclic fundamental groups}

\author[R. \.{I}. Baykur]{R. \.{I}nan\c{c} Baykur}
\address{Department of Mathematics and Statistics, University of Massachusetts, Amherst, MA 01003, USA}
\email{inanc.baykur@umass.edu}

\author[A. I. Stipsicz]{Andr\'{a}s I. Stipsicz}
\address{HUN-REN R\'enyi Institute of Mathematics\\
H-1053 Budapest\\ 
Re\'altanoda utca 13--15, Hungary}
\email{stipsicz.andras@renyi.hu}

\author[Z. Szab\'o]{Zolt\'an Szab\'o}
\address{Department of Mathematics\\
Princeton University\\
 Princeton, NJ, 08544}
\email{szabo@math.princeton.edu}

\begin{abstract}
For each nonnegative integer $m$ we show that any closed, oriented
topological four-manifold with fundamental group $\Z_{4m+2}$ and odd
intersection form, with possibly seven exceptions, either admits no
smooth structure or admits infinitely many distinct smooth structures
up to diffeomorphism. Moreover, we construct infinite families of
non-complex irreducible fake projective planes with diverse
fundamental groups.
\end{abstract}
\maketitle

\section{Introduction}
\label{sec:intro}

Since the advent of diffeomorphism invariants coming from gauge
theory, many topological four-manifolds have been shown to support
infinitely many, pairwise non-diffeomorphic, smooth structures. Indeed, no
smoothable four-manifold is known to admit only finitely many smooth
structures. One might expect this to be a general phenomenon, at least
for four-manifolds with certain fundamental groups.
Our main result underscores this expectation:

\begin{thm}
\label{thma} 
Let $X$ be a closed, oriented topological four-manifold with cyclic
fundamental group $\Z_{4m+2}$ for some nonnegative integer $m$, and
odd intersection form. Unless $b_2^+(X)=b_2^-(X) \leq 7$, the
four-manifold $X$ admits either no smooth structure or infinitely many
distinct smooth structures.
\end{thm}

Recall that the intersection form $Q_X$ of a closed, oriented
four-manifold $X$ is \emph{odd} if the pairing $Q_X(\alpha , \alpha )=
\langle \alpha \cup \alpha, [X]\rangle$ is odd for some $\alpha \in
H^2(X; \Z)$, where $[X]$ is the orientation class; $Q_X$ is
\emph{even} otherwise. Odd intersection forms of smoothable
four-manifolds with cyclic fundamental group are isomorphic to
$a\,\langle 1 \rangle \oplus b\,\langle -1 \rangle$ for $a=b_2^+(X)$
and $b=b_2^-(X)$ with $a+b>0$, \cite{donaldson:definite1, donaldson:definite2}.

By the homeomorphism classification in our case (to be reviewed in
$\S$\ref{sec:topclass}), Theorem~\ref{thma} leaves out only seven
smoothable topological four-manifolds for each fixed fundamental group
$\Z_{4m+2}$, for which we currently do not know the existence of more
than one smooth structure (cf. also Remark~\ref{rk:remaining}).

In contrast, no distinct smooth structures are known on infinitely many \emph{simply connected} four-manifolds with intersection forms $a\,\langle 1 \rangle \oplus b\,\langle -1 \rangle$ with even $a, b$, and to date, no exotic smooth structures are known on simply
connected four-manifolds with $ab=0$.
As we focus on four-manifolds with cyclic fundamental groups, we
handle more challenging cases 
by analyzing the smooth structures on their finite covers. 
Our work in this direction therefore hinges on
equivariant constructions of four-manifolds with non-trivial
Seiberg-Witten invariants and orientation-preserving free
involutions---building on our earlier works in
\cite{baykur-hamada:signaturezero, stipsicz-szabo:definite}.

Based on our Theorem~\ref{thma}, and accompanying observations (see
Remark~\ref{rk:irred}), we are prompted to pose a broader question:

\begin{question} 
Does every closed, oriented topological four-manifold with a finite
cyclic fundamental group admit either no or infinitely many smooth
structures?
\end{question}

\smallskip
A four-manifold $X$ is said to be an \emph{exotic} copy of another
four-manifold $Z$ if $X$ and $Z$ are homeomorphic but not
diffeomorphic.  Theorem~\ref{thma} translates to any smooth $Z$ in the
listed homeomorphism classes having infinitely many exotic copies. If
$X$ and $Z$ are closed, smooth four-manifolds with the same rational
cohomology ring and with isomorphic integral intersection forms, 
but are not diffeomorphic, we say that
$X$ is a \emph{fake} copy of $Z$.

A relatively recent result in four-dimension topology is the complete
classification of complex \emph{fake projective planes}, namely, fake
$\CP$s that admit complex structures.
The first example was found by
Mumford \cite{mumford}, with further examples in \cite{ishida-kato, keum}. Building on the classification of Prasad and
Yeung \cite{PrasadYeung}, computer-assisted calculations in arithmetic
geometry established that there are exactly 
$50$ smooth fake projective planes that admit a complex structure
\cite{CartwrightSteger}. These complex surfaces are smoothly
irreducible; see Proposition~\ref{prop:FPPirreducible}. Recall that a smooth four-manifold is \emph{irreducible}
if it is not a smooth connected sum of two four-manifolds neither of
which is a homotopy four-sphere.

Motivated by this array of results, Fintushel and Stern raised the
question of whether there were irreducible fake projective planes
besides the complex ones and what one could tell about their
fundamental groups \cite{Stern:CornellTalk}.  Our second theorem
provides insights via constructions of non-symplectic fake $\CP$s.

\begin{thm}\label{thmb} 
There is an infinite family of pairwise non-diffeomorphic, irreducible
fake projective planes with the same fundamental group $G$, such that
$G$ can be taken to be $R \x \Z_2$, where $R$ is the fundamental group
of any rational homology three-sphere, or $\Z_m \x D_m$, for any odd
$m$, where $D_m$ is the dihedral group of order $2m$.  None of these
fake projective planes admit symplectic or complex structures.
\end{thm}

 Reducible fake four-manifolds are easy to come
by; the connected sum $X$ of a four-manifold $Z$ with any rational
homology four-sphere is a fake copy of $Z$. We obtain our irreducible
fake projective planes through a variety of equivariant constructions,
where we distinguish the smooth structures and detect
their irreducibility through Seiberg-Witten invariants of their double
covers, after orientation-reversal. To the best of our knowledge,
these non-complex irreducible fake projective planes, along with the
examples with $\pi_1=\Z_2$ in \cite{stipsicz-szabo:definite} subsumed
here, are the first of their kind.

The fake projective plane fundamental groups we obtain in this paper
are all semi-direct products with $\Z_2$. Otherwise, the groups
mentioned in Theorem~\ref{thmb} constitute a fairly large collection
of groups, spanning over infinitely many finite abelian, finite
non-abelian, and infinite groups. Indeed, we get non-complex
irreducible fake projective planes with first integral homology group
realizing any prescribed finite abelian group with a $\Z_2$ factor in
its primary decomposition. As our way of constructing
irreducible fake projective planes falls short of yielding symplectic
examples, we note the following natural questions:
\begin{question} 
Which groups are realized as fundamental groups of irreducible fake
projective planes? Are there only 50 groups that can arise as
fundamental groups of {\emph {symplectic}} fake projective planes?
\end{question}

\medskip

\noindent {\textit{Outline of the paper}}: We review the
topological classification of four-manifolds with finite cyclic
fundamental group in $\S$\ref{sec:topclass}. General equivariant construction
schemes for four-manifolds with free involutions via normal connected
sum and
circle sum operations, respectively, and a few key lemmas regarding
their algebraic and smooth topology are given in
$\S$\ref{sec:fibersum} and $\S$\ref{sec:circlesum}. We build our
irreducible four-manifolds with \mbox{$\pi_1=\Z_{2}$, } odd intersection
form and signatures $\sigma=0$ and $\pm1$, in $\S$\ref{sec:signzero}
and $\S$\ref{sec:signone}. Expanding on these constructions, we prove
Theorem~\ref{thma} in $\S$\ref{sec:proofA}. Finally, in
$\S$\ref{sec:fpp}, we present our constructions of non-complex fake
projective planes and prove Theorem~\ref{thmb}.

\vspace{0.1in} 
\noindent {\textit{Acknowledgements}}: RIB was partially
supported by the NSF grant DMS-2005327, AS was partially supported by
the NKFIH Grant K146401 and by ERC Advanced Grant KnotSurf4d, and SzZ
was partially supported by the
Simons Grant \emph{New structures in low dimensional topology.} The authors
would like to thank the organizers of the 2023 \textit{Exotic
  4-manifolds} workshop at Stanford University, supported by the
Simons Collaboration Grant \textit{New structures in low dimensional
  topology}, for an extremely stimulating event, where this
collaboration started.

\smallskip
\section{Background results} \label{sec:background}

In this section we gather some preliminary results that will be repeatedly used
throughout the paper.

\subsection{Topological classification of four-manifolds with $\pi_1=\Z_n$} \label{sec:topclass} \

By the works of Freedman~\cite{freedman}, and Hambleton and Kreck 
\cite{hambleton-kreck2, hambleton-kreck1, hambleton-kreck3}, the
homeomorphism type of a closed, oriented, topological four-manifold
$X$ with fundamental group $\Z_n$ is determined by its
Kirby-Siebenmann invariant, intersection form and the $w_2$--type. The
$w_2$--type is determined by whether or not $X$ and its universal
cover $\widetilde{X}$ are spin; it might be that (i) both are
non-spin, (ii) both are spin,  or (iii) $X$ is non-spin, but
$\widetilde{X}$ is spin.
The intersection form of $X$ is odd in case (i) and even in cases
(ii) and (iii).

\begin{figure}[htb]
	\centering
	\includegraphics[height=110pt]{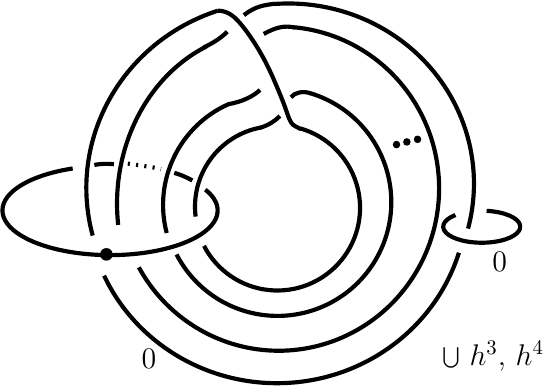}
	\caption{The Kirby diagram of the smooth rational homology
          four-sphere $L_n$. The middle curve wraps around $n$
          times. }
	\label{fig:rat_hom_S4new}
\end{figure}

Let $L_n$ be the rational homology four-sphere given by the Kirby
diagram in Figure~\ref{fig:rat_hom_S4new}, which is the spun of the
lens space $L(n,1)$ \cite{plotnick}. Recall that the spun $S(Y)$ of a
three-manifold $Y$ is the four-manifold $(S^1\times (Y \, \setminus \,
D^3))\cup D^2\times S^2$, where we glue the two-handle of $D^2\times
S^2$ along a circle $S^1\times \{ p\}$ with framing given by
$S^1\times \{ q\}$, where $p,q$ are boundary points of the disk $D^3$
deleted from $Y$.

\begin{remark}
Any even framing on the middle $2$--handle in
 Figure~\ref{fig:rat_hom_S4new}  gives  $L_n$. Whereas, an odd framing yields the Kirby diagram of the
  companion rational homology four-sphere $L'_n$, the
  \emph{twist-spun} of $L(n,1)$;
  cf. \cite[Figure~11]{baykur-kamada}. Up to homeomorphism, these are
  the only two closed, oriented, topological manifolds with $\pi _1=\Z _n$
  and $b_2=0$; $L_n$ is spin, while $L_n'$ is not.
  It is easily seen by Kirby calculus that $L_n\# \CP$ and $L'_n\#\CP$ are
  diffeomorphic.
  \end{remark}

By the above topological classification result,
combined with Donaldson's diagonalizability theorem
\cite{donaldson:definite1, donaldson:definite2}, we conclude that when the intersection form is definite,
smooth four-manifolds of
$w_2$-type (i) with the intersection form $a\,\langle 1 \rangle \oplus
b\,\langle -1 \rangle$ and fundamental group $\Z _n$
are homeomorphic to the smooth
four-manifold
\begin{equation*}
R_{a,b}(n):=L_n \# a \, \CP \# b\, \CPb, \text{ for } a, b \in \N \text{ with }a+b>0 .
\end{equation*}
When $n=2$, the case we are going to deal with the most, we write
$R_{a,b}:=R_{a,b}(2)$.

\subsection{A $\Z_2$--equivariant construction via generalized normal connected sum} \label{sec:fibersum}  \

Hereon, $\Sigma_g^b$ denotes a compact, connected, oriented surface of genus $g$ and $b$ boundary components, and we write $\Sigma_g$ when $b=0$.

Let $Z$ be a closed, oriented, smooth four-manifold with an embedded,
orientable surface $\Sigma$ of genus $g \geq 1$ and self-intersection zero. Let
$\nu \Sigma$ denote a neighborhood of $\Sigma$ isomorphic to its
normal (trivial) disk bundle in $Z$. Let $\phi$ be the non-trivial deck
transformation of a double cover $\Sigma_g \to \#_{g+1}
\RP^2$. Given an identification $\eta \colon \nu \Sigma \to D^2 \x
\Sigma_g$, that is, a framing of $\Sigma$, we define the
orientation-reversing self-diffeomorphism
\begin{equation}\label{eq:gluing}
  \Phi:=(\eta|_{\partial (\nu \Sigma)})^{-1} \circ
  (\textrm{id}_{S^1} \x \phi)\circ \eta|_{\partial (\nu \Sigma)}
\end{equation} 
of $\pm \partial (\nu \Sigma)$, which is a free involution. Then, the four-manifold
\[ \widetilde{X}:= (Z \setminus \nu \Sigma)\cup_\Phi (Z \setminus \nu \Sigma) \, , \]
which we are going to refer as the \emph{double of $(Z, \Sigma )$},
admits a smooth, orientation-preserving, free involution $\tau$ that
swaps the two copies of $Z \setminus \nu \Sigma$, while restricts to
$\Phi$ on their identified boundaries in $\widetilde{X}$. The quotient
is a smooth, closed, oriented four-manifold $X:= \widetilde{X} \, /
\tau$. Note that the result might depend on the chosen framing and
the involution $\phi$ as well.

The next few lemmas explain how  to calculate various topological
invariants of the quotient four-manifold $X$ in terms of those of $(Z,
\Sigma)$ we started with.  Here $\lceil x \rceil \in \Z$ denotes the
ceiling value of $x \in \R$.

\begin{lemma}\label{lem:pi1}
Let $\pi_1(Z\setminus \, \Sigma)$ be abelian. Assume that the
inclusion induced homomorphism $H_1(\Sigma; \Z) \to H_1(Z\, \setminus
\Sigma; \Z)$ is onto and its kernel contains at least $\lceil
\frac{g}{2}\rceil$ dual pairs in a symplectic basis of
$H_1(\Sigma; \Z)$. Then, we can choose $\eta$ in~\eqref{eq:gluing} so that
$\pi_1(X)=\Z_2$. When $\pi_1(Z \setminus \, \Sigma) =1$, for any
choice of $\eta$ we get $\pi_1(X) = \Z_2$.
\end{lemma}

\begin{proof}
Identify the $0$--section $\Sigma$ of $\nu \Sigma$ with $\Sigma_g = \{ 0 \} \x \Sigma_g$ via any $\eta_0  \colon \nu \Sigma \to D^2 \x \Sigma_g$. Under this identification, let $\{a'_i, b'_i\}$ denote the corresponding symplectic basis of $H_1(\Sigma_g; \Z)$ and let $K \leq H_1(\Sigma_g; \Z)$ correspond to the kernel of $H_1(\Sigma;\Z) \to H_1(Z \setminus \Sigma; \Z)$. 

There is a symplectic basis $\{a_i, b_i\}$ on $\Sigma_g$ where the
action of the free involution $\phi$ on $H_1(\Sigma_g; \Z)$ is such that
$\phi_*(a_i)= \pm a_{g-i}$ and $\phi_*(b_i)= \pm b_{g-i}$, for all $i=
1, \ldots \lceil \frac{g}{2}\rceil$. Recall that, by our hypothesis,
there are $\lceil \frac{g}{2}\rceil$ pairs among $\{a'_i, b'_i\}$
contained in $K \subset H_1(\Sigma_g; \Z)$. There exists a $\psi \in
\Diff(\Sigma_g)$ which maps these pairs in $K$ to the first $\lceil
\frac{g}{2}\rceil$ pairs $\{a_1, b_1, \ldots, a_{\lceil
  \frac{g}{2}\rceil}, b_{\lceil \frac{g}{2}\rceil}\}$. The elements in
$K \cup (\psi^{-1}\phi \,\psi)_*(K)$ generate all of $H_1(\Sigma_g; \Z)$.

Post-composing $\eta_0$ with $\textrm{id} \x \psi$, we get a new
$\eta$, and we can then take the corresponding gluing map $\Phi$
defined as in \eqref{eq:gluing}. Let $\Sigma'= \eta^{-1}( \{1 \} \x
\Sigma_g) \subset \partial (\nu \Sigma)$ be a push-off of $\Sigma$
under the framing induced by $\eta$. Because $\pi_1(Z\setminus
\Sigma)$ is abelian, the surjectivity of $H_1(\Sigma; \Z) \to
H_1(Z\setminus \Sigma; \Z)$ implies that the inclusion induced
homomorphisms $\pi_1(\Sigma') \to \pi_1(Z\setminus \nu \Sigma)$, and
in turn, $\pi_1(\partial(Z \setminus \nu \Sigma)) \to \pi_1(Z\setminus
\nu \Sigma)$ are also surjective. Then, by the Seifert-Van Kampen theorem,
\[ \pi_1(\widetilde{X})= \pi_1(Z \setminus \nu \Sigma)\underset{\pi_1(\partial(Z\setminus \nu \Sigma))}{*} \pi_1(Z \setminus \nu \Sigma) \]
is abelian, and by the above observations, it is equal to the quotient of $\pi_1(\Sigma')$ by $N'=(\eta^{-1})_*(\langle K \cup (\psi^{-1}\phi \,\psi)_*(K) \rangle$. So $\pi_1(\widetilde{X})=1$, which implies that $\pi_1(X)=\Z_2$.

When $\pi_1(Z \setminus \Sigma)=1$, both hypotheses on $H_1(\Sigma; \Z) \to H_1(Z \setminus \Sigma; \Z)$ hold vacuously, and since $K=H_1(\Sigma_g; \Z)$ in this case, any choice of $\eta$ will work.
\end{proof}

\smallskip
\begin{lemma}\label{lem:b2}
The Euler characteristic and the signature of $X$ are given by
\[ \chi(X)= \chi(Z)+ 2(g-1) \, \text{ and } \, \sigma(X)=\sigma(Z) \, . \]
When $Z \setminus \nu \Sigma$ contains an immersed surface of odd self-intersection, so does $X$.
\end{lemma}

\begin{proof}
From the decomposition $Z= (Z \setminus \nu \Sigma) \cup \nu \Sigma$,
we get $\chi(Z)= \chi(Z \setminus \nu \Sigma) +2-2g$ and
$\sigma(Z)=\sigma(Z \setminus \nu \Sigma)$. Then, from the
decomposition $\widetilde{X}= (Z\setminus \nu \Sigma) \cup (Z\setminus
\nu \Sigma)$, we obtain $\chi(\widetilde{X})=2\chi (Z)+4(g-1)$ and
$\sigma(\widetilde{X})=2 \sigma(Z)$. Finally, because the Euler
characteristic and the signature are multiplicative under finite
covers, we get $\chi(X)=\frac{\chi(\widetilde{X})}{2}= \chi(Z)+
2(g-1)$ and $\sigma(X)=\frac{\sigma(\widetilde{X})}{2}=\sigma(Z)$, as
claimed.  As $Z\setminus \nu \Sigma$ embeds into $X$, the second claim
follows at once.
\end{proof}

\smallskip
\begin{lemma}\label{lem:sympsum}
If $Z$ admits a symplectic form $\omega$ with respect to which $\Sigma$ is a symplectic submanifold, then the double $\widetilde{X}$ admits a symplectic structure $\widetilde{\omega}$, which pulls-back to $\pm \omega$
under the inclusions of the two copies of $Z \setminus \nu \Sigma$ into $\widetilde{X}$. If $Z \setminus \nu \Sigma$ does not contain any exceptional spheres, then $X$ is minimal. If also $\pi_1(Z\setminus \Sigma)$ is finite, or more generally, if $\pi_1(\widetilde{X})$ is residually finite, then $X$ is irreducible. 
\end{lemma}

\begin{proof}
Let $\Sigma$ be a symplectic submanifold of $(Z, \omega)$, so it is
oriented by $\omega$. Then $\overline{\Sigma}$, which denotes the
oppositely oriented surface $\Sigma$, is symplectic with respect to
the symplectic form $-\omega$ on $Z$. We can thus take a
\emph{symplectic normal connected sum} \cite{gompf:fibersum} of $(Z, \omega)$ and
$(Z, -\omega)$ along $\Sigma$ and $\overline{\Sigma}$ to get a
symplectic four-manifold $(\widetilde{X}, \widetilde{\omega})$, where
$\widetilde{\omega}$ restricts to $\pm \omega$ on each copy of $Z
\setminus \nu \Sigma$ as claimed.

If $Z \setminus \nu \Sigma$ does not contain any exceptional spheres,
then $\widetilde{X}$ is minimal by
\cite[Theorem~1.1]{usher:minimality}.
When $\pi_1(\widetilde{X})$ is residually finite, $\widetilde{X}$ is
also smoothly irreducible by
\cite[Corollary~1.4]{hamilton-kotschick}. Any amalgamated free product
of finite groups is residually finite by \cite[Theorem~2]{baumslag},
so by the Seifert-Van Kampen theorem, $\pi_1(\widetilde{X})$ is
residually finite when $\pi_1(Z \setminus \nu \Sigma)$ is finite.
Minimality or irreducibility of the cover $\widetilde{X}$ then
implies the same for $X$.
\end{proof}

We will often refer to the four-manifold $Z$ or $(Z, \omega)$, together
with the surface $\Sigma \subset Z$ as the \emph{inputs} of this
equivariant normal connected sum construction, and the quotient four-manifold $X=
\widetilde{X} / \tau$ as the \emph{output}.

\begin{remark}
There are other gluing maps one can employ in our $\Z_2$--equivariant
normal connected sum construction, by replacing the orientation-reversing
self-diffeomorphism $\textrm{id}_{S^1} \x \phi$ of $S^1 \x \Sigma_g$
we had with $a\x \psi$, where $a$ is the antipodal map on $S^1$ and
$\psi$ is \emph{any} orientation-reversing diffeomorphism of
$\Sigma_g$. In particular, we can take $\psi$ to be a diffeomorphism
of $\Sigma_{2k}=\Sigma_k^1 \cup \Sigma_k^1$ that switches the two
subsurfaces $\Sigma_k^1$ by a reflection along their identified
boundaries.  It is straightforward to see that all the lemmas above
also apply in this case.
\end{remark}

\smallskip
\subsection{A $\Z_2$--equivariant construction via circle sum} \label{sec:circlesum} \

We present another general construction to derive smooth four-manifolds, as
quotients of four-manifolds with non-trivial Seiberg-Witten invariants
under free involutions, this time via the \emph{circle sum} operation.

Let $\widetilde{X}$ be a closed, simply connected, oriented, smooth
four-manifold with an orientation-preserving, free involution
$\tau\colon \widetilde{X} \to \widetilde{X}$.
As before, let $q\colon
\widetilde{X} \to X={\widetilde {X}}/\tau$ be the induced double
cover. Take a lift $\alpha$ of a simple closed curve generating
$\pi_1(X) = \Z_2$, so $\gamma:=\alpha \cup \tau(\alpha)$ is a
$\tau$--equivariant simple closed curve in $\widetilde{X}$.  Let $\nu
(\gamma) \subset X$ be a $\tau$--invariant tubular neighborhoood of
$\gamma$; that is, there is an identification $\nu(\gamma) \cong S^1
\x D^3 \subset \C \x D^3$, under which the $\tau$ action is given by
$(t, x) \mapsto (-t, x)$.

Let $Y$ be a closed, oriented three-manifold. The product four-manifold
$S^1 \times Y$ admits an orientation-preserving free involution
$\tau'(t,x)=(-t, x)$. For \mbox{$\gamma':=S^1 \x \{y_0\}$} in $S^1 \x Y$, let
$\nu(\gamma')= S^1 \x D \subset S^1 \x Y$ be a tubular neighborhood of
$\gamma'$ for some embedded $3$--disk $D$ centered at $y_0 \in
Y$. Then $\nu(\gamma')$ is $\tau'$--invariant, and the restriction of
$\tau'$ to $\nu(\gamma') \cong S^1 \x D^3$, under this identification,
is also given by $(t, x) \mapsto (-t, x)$.

We take the circle sum of $(\widetilde{X}, \gamma)$ and $(S^1 \times
Y, \gamma')$ by a fiber-preserving diffeomorphism $\partial
\nu(\gamma) \to \partial \nu(\gamma')$. We get a free involution
$\widetilde{\tau}$ on the circle sum $ \widetilde{X}
\#_{\gamma=\gamma'} (S^1 \x Y)$, which extends the involutions $\tau$
and $\tau'$ on the summands.
Because $\gamma \subset \widetilde{X}$ is null-homotopic (as
$\widetilde{X}$ is simply connected), we have that
$ \widetilde{X} \#_{\gamma=\gamma'} (S^1 \x Y)$
is diffeomorphic to the connected sum
$ \widetilde{X} \# S(Y)$, where $S(Y):=(S^1 \x Y\setminus D)
\cup_{\textrm{id}} (D^2 \x S^2)$ is the spun four-manifold of $Y$. Note that we have
$\pi_1(\widetilde{X} \# S(Y))\cong \pi _1 (S(Y))\cong \pi_1(Y)$.

We define the closed, oriented, smooth four-manifold $X_Y$ as
$\widetilde{X}\# S(Y) / \widetilde{\tau}$.

\begin{lemma}\label{lem:circlesum-topology}
The fundamental group of $X_Y$ is given by
\[ \pi_1(X_Y) = \pi_1(Y) \x \Z_2 \, ,\]
and its Euler characteristic and the signature are
\[ \chi(X_Y)= \frac{\chi(\widetilde{X})}{2} =\chi (X)\ \text{ and } \
\sigma(X_Y)=\frac{\sigma(\widetilde{X})}{2} =\sigma (X) \, .\]
\end{lemma}

\begin{proof}
Let $q_1\colon \widetilde{X}\# S(Y) \to X_Y$ be the corresponding double
cover. The induced short exact sequence
\[ 
1 \longrightarrow \pi_1(\widetilde{X} \# S(Y)) \overset{{q_1}_*}{\longrightarrow} \pi_1(X_Y) \longrightarrow \Z_2 \longrightarrow 1 
\]
splits on the right by the homomorphism $s\colon \Z_2 \to \pi_1(X_Y)$
defined by $s(1)=\gamma_1$, where 
$\gamma_1=q_1(\alpha_1)$ for some
push-off $\alpha_1$ of the arc $\alpha \subset \widetilde{X}\# S(Y)$ to
$\partial \nu(\gamma)$ so that \mbox{$\alpha_1 \cup \tau(\alpha_1)$} is a
loop doubly covering $\gamma_1\subset X_Y$.  It follows that $\pi_1(X_Y)$ is a
semi-direct product $\pi_1(\widetilde{X}\# S(Y)) \rtimes \Z_2$, and in
fact, the action of $\Z_2$ on ${q_1}_*(\pi_1(\widetilde{X}\# S(Y)))
\triangleleft
\pi_1(X_Y)$ is trivial: for every $\beta$ in
${q_1}_*(\pi_1(\widetilde{X}\# S(Y)))={q_1}_*(\pi_1(Y))$, we have
$\gamma_1 \beta \gamma_1^{-1} =q_1 (\alpha_1 \tilde{\beta}
\,\tau(\alpha_1))= q_1(\tilde{\beta})= \beta$, where $\tilde{\beta}$
is the lift of $\beta$. Hence
\[
\pi_1(X_Y )= \pi_1(\widetilde{X}\# S(Y)) \x \Z_2=\pi_1(Y) \x \Z_2.
\]

We have $\chi(S^1 \x Y)=\sigma(S^1 \x Y)=0$ and $\chi(\nu
\gamma')=\sigma(\nu \gamma')=0$. Therefore, from the decomposition
$(\widetilde{X} \,\setminus \nu (\gamma)) \cup (S^1 \x Y\, \setminus
\nu (\gamma'))$, we get $\chi(\widetilde{X}\#
S(Y))=\chi(\widetilde{X})$ and $\sigma(\widetilde{X}\#
S(Y))=\sigma(\widetilde{X})$. We then get $\chi(X_Y )=
\frac{\chi(\widetilde{X})}{2}=\chi (X)$ and
$\sigma(X_Y )=\frac{\sigma(\widetilde{X})}{2}=\sigma (X)$, once again by the
multiplicativity of the Euler characteristic and the signature for
finite covers.
\end{proof}

We are going to refer to the four-manifold with involution
$(\widetilde{X}, \tau)$ and the three-manifold $Y$ as the
\emph{inputs} of this equivariant circle sum construction, and the
resulting quotient four-manifold $X_Y = (\widetilde{X}
\#_{\gamma=\gamma'} (S^1 \x Y)) / \widetilde{\tau}$ as the
\emph{output}.

\smallskip
\subsection{Torus surgeries} \label{sec:SWandtorus} \

We briefly review surgeries along tori in four-manifolds and their
effect on the Seiberg-Witten invariants.  See also
\cite{fintushel-stern:reverse, stipsicz-szabo:definite} for a more
complete discussion.

Let $X$ be a smooth four-manifold and $T\subset X$ be an oriented,
embedded torus with trivial normal bundle.  A small disk neighborhood
$\nu T$ is diffeomorphic to $T^2\times D^2$; let $\eta\colon \nu T \to
T^2\times D^2$ be a framing, that is, a trivialization of this normal
bundle.  An oriented loop $\lambda \subset T$ lifts to a push-off
$\lambda _{\eta} \subset \partial ( \nu T)$ via $\eta$.  As customary,
we let $\mu _T$ denote the oriented meridian of $T$, viewed also in $\partial (\nu T)$.

Having these data fixed, for any $\frac{p}{q} \in \Q \cup \{\infty\}$
we can define a new four-manifold $X':=(X, T,\eta, \lambda ,
\frac{p}{q})$ as the result of  \emph{$\frac{p}{q}$-surgery on
$X$ along $T$ with framing} $\eta$, \emph{surgery curve} $\lambda$ and
\emph{surgery coefficient } $\frac{p}{q}$ as follows: we remove ${\rm
  {int}}\, \nu T$ from $X$ and glue back $T^2\times D^2$ by a
diffeomorphism $\varphi$ which maps the circle $\partial (\{ \text{pt}
\} \x D^2) \subset \partial (T^2\times D^2)$ to a simple closed curve
representing the homology class $p[\mu _T] + q [\lambda_\eta]$ in
$H_1(\partial (\nu T); \Z)$, which is unique up isotopy.  One easily
concludes by the Seifert-Van Kampen theorem that (choosing the base point on $\partial (\nu T)$) we have
\begin{equation*}
\pi_1(X') = \pi_1(X \, \setminus \nu T) \, /\, 
\langle [\mu_{T}]^p [\lambda_\eta]^q =1\rangle.
\end{equation*}

This generalization of Dehn surgery on three-manifolds to the
four-dimensional setting is called \emph{torus surgery}.  If $X$
admits a symplectic form $\omega$ and $T\subset (X, \omega )$ is a
Lagrangian torus, then the torus $T$ admits a canonical framing
$\eta_L$, called the \emph{Lagrangian framing}, distinguished by the
property that images of $T^2\times \{ \text{pt}\}\subset T^2\times
D^2$ under this framing are all Lagrangian in $(X, \omega)$.  For any
$k \in \Z$, a \mbox{$\frac{1}{k}$--surgery} on $(X, \omega)$ along a
Lagrangian torus $T$, taken with the Lagrangian framing, for
\emph{any} surgery curve yields another symplectic four--manifold
$(X', \omega')$, where $\omega'$ agrees with $\omega$ away from $T$
\cite{ADK, Luttinger}. This particular type of torus surgery is called a
\emph{Luttinger surgery}.

When we perform torus surgery along a framed torus $T \subset X$, we
will often use the shorthand notation $(T, \lambda, \frac{p}{q})$
suppressing the four-manifold and the framing from the
notation. Whenever we spell out that $T$ is Lagrangian and we are to
perform a Luttinger surgery, the framing we take is understood to be
the Lagrangian framing.  At times, we will denote the four-manifold
obtained by torus surgeries along Lagrangian tori $T_1, \ldots, T_n$
in a four--manifold $X$ with specified surgery curves $\lambda_1,
\ldots, \lambda_n$ by $X(m_1, \ldots, m_n)$, where $m_i$ denote the
surgery coefficients.

\smallskip
\subsection{Seiberg-Witten invariants} \label{sec:SW} \

The \emph{Seiberg-Witten} function of a closed,
oriented, smooth four-manifold $X$ is a map
\[
\SW_X\colon Spin^c(X) \to \Z ,
\]
where $Spin^c(X)$ denotes the set of spin$^c$ structures on $X$.  For
four-manifolds having no $\Z_2$--torsion
in $H_1(X; \Z )$, the set of their spin$^c$
structures can be identified with characteristic
cohomology classes in $H^2(X; \Z)$ through their first Chern class.

The Seiberg-Witten value $\SW_X( {\mathfrak {s}})$  counts solutions of a partial differential equation associated
to the  spin$^c$ structure ${\mathfrak {s}}$, a metric on $X$, and a
perturbation of the partial differential equation.
If the four-manifold $X$ with $b_1(X)=0$ satisfies  $b_2^+(X)>1$,
then $\SW_X({\mathfrak {s}})$ can
be shown to be independent of the chosen metric and perturbation.
In this case, the map $\SW_X$ is (up to sign) a diffeomorphism
invariant, that is, for a diffeomorphism $f\colon X_1\to X_2$, we have
\[
\SW_{X_2}({\mathfrak {s}})=\pm \SW_{X_1}(f^*({\mathfrak {s}}))
\]
for every ${\mathfrak {s}} \in Spin^c(X_2)$. If $X$ has $b_2^+(X)=1$,
then the metrics and perturbations are partitioned into chambers, and
the value of the Seiberg-Witten invariant on a spin$^c$ structure may
depend on the chosen chamber.

A cohomology class $K\in H^2 (X; \Z )$ is a Seiberg-Witten \emph{basic
class} if there is ${\mathfrak {s}} \in Spin^c(X)$ so that
$c_1({\mathfrak {s}})=K$ and $\SW_X({\mathfrak {s}})\neq 0$.  For a
fixed four-manifold $X$ there are finitely many spin$^c$ structures
with non-vanishing Seiberg-Witten invariants, hence a given $X$ admits
finitely many basic classes. The four-manifold $X$ is of
Seiberg-Witten \emph{simple type} if for any basic class $K$, the
equality $K^2=2\chi (X) +3\sigma (X)$ holds.

The Seiberg-Witten function $\SW_X$ has finite support and admits the
symmetry $\SW_X({\overline {\mathfrak {s}}})=\pm \SW_X({\mathfrak
  {s}})$, where ${\overline {\mathfrak {s}}}$ denotes the spin$^c$
structure conjugate to ${\mathfrak {s}}$. In addition, the support of
$\SW_X$ is connected to the geometry of the underlying four-manifold
through the \emph{adjunction inequality}:
\begin{proposition}(\cite{KMcikk, KMkonyv}) \label{prop:adjunction}
Let $\Sigma \subset X$ be an embedded, closed, oriented surface of genus
$g(\Sigma ) \geq 1$ and self-intersection $[\Sigma ]^2 \geq 0$. For any basic class $K\in H^2(X; \Z )$
of the four-manifold $X$ with $b_2^+(X)>1$ we have
\[
2g(\Sigma )-2\geq [\Sigma ]^2+\vert K([\Sigma ])\vert .
\]
\end{proposition}
If the four-manifold $X$ with $b_2^+(X)>1$
admits a symplectic structure $\omega$,
which naturally induces a spin$^c$ structure with Chern class $c_1(X, \omega)$, then by \cite{taubes},
we have
\begin{equation}\label{eq:symp}
  \SW_X(\pm c_1(X, \omega ))=\pm 1 . \, 
\end{equation}
Symplectic four-manifolds with $b_2^+(X)>1$ are of simple type \cite{taubes}.

A simple numerical invariant of smooth four-manifolds can be defined
by taking the maximal Seiberg-Witten value over the 
spin$^c$ structures:
\[
M_{\SW}(X)=\max \{ \vert \SW_X(\s )\vert \mid \s \in {\rm {Spin}}^c(X)\} .
\]
$M_{\SW}(X)$ is a diffeomorphism invariant of $X$ when $b_2^+(X)>1$.
If the four-manifold $X$ has $b_2^+(X)=1$, a similar invariant can be
defined.  In this case, however, the value of the Seiberg-Witten
function might depend on the chosen metric/perturbation. The space of
these choices partition into \emph{chambers} on which the value of
$SW_X$ is constant, and the wall-crossing formula determines the
change in the value when we traverse from one chamber to another.  For
a chamber ${\mathfrak {c}}$ and metric/perturbation $(g, \delta )$
representing that chamber, we define the value
\[
M_{\SW}^{\mathfrak {c}}(X)=\max \{ \vert \SW_X(\s , g, \delta )\mid \s \in
{\rm {Spin}}^c(X) .\}
\]
The wall-crossing formula implies that the maximum $M_{\SW}(X)$ of
$M_{\SW}^{\mathfrak {c}}(X)$ for all chambers is a diffeomorphism
invariant of $X$.

\smallskip
Below, we quickly review how Seiberg-Witten invariants interact with
various topological constructions of four-manifolds.

The invariants of connected sums are determined by the Seiberg-Witten
invariants of the summands. Assume that $X=X_1\#X_2$. If
$b_2^+(X_i)>0$ for $i=1,2$ then $\SW_X\equiv 0$. Assume therefore that
$b_2^+(X_2)=0$ and for simplicity let us also assume that $b_1
(X_2)=0$.  By Donaldson's diagonalizability theorem, the intersection
form $Q_{X_2}$ is isomorphic to $b_2(X_2)\langle -1\rangle$, that is,
$X_2$ is a rational homology $\#_b \CPb$.  Suppose that the second
homology classes providing the basis of $H_2(X_2; \Z )/Tor$ on which
the matrix is diagonal are denoted by $\{ e_1, \ldots, e_b\}$, where
$b=b_2(X_2)$.
In this case, the \emph{blow-up formula} relates the Seiberg-Witten
invariants of $X$ and $X_1$:

\begin{proposition}[\cite{FSImmersed}] \label{prop:blowup-SW}
  Suppose that ${\mathfrak {s}}_i\in Spin^c (X_i)$ are spin$^c$ structures
  on $X_1$ and $X_2$, and
  let $n_1, \ldots , n_b$ be  non-negative integers
  given by $\langle c_1({\mathfrak {s}}_2), e_i\rangle =2n_i+1$.
  Assume, furthermore,
that these integers are constrained by
\[
\frac{1}{4}(c_1({\mathfrak {s}}_1)^2-2\chi (X_1)-3\sigma (X_1))-\sum
_{i=1}^b n_i(n_i+1) \geq 0.
\]
Then
\[
\SW_{X_1}({\mathfrak {s}}_1)=\pm \SW_X({\mathfrak {s}}_1\# {\mathfrak {s}}_2).
\]
\end{proposition}

A simple consequence of the above formula is that
\begin{corollary}\label{cor:BlowUpDistinction}
If two homeomorphic smooth four-manifolds $X_i$, $i=1,2$, 
with $b_2^+>1$ are  distinguished by $M_{\SW}(X_1)\neq M_{\SW}(X_2)$, then
the same holds for the pair $X_i\# \ell \CPb$, $i=1,2$, for any $\ell \in \N$,
that is, then
\[
M_{\SW}(X_1\# \ell \CPb )\neq M_{\SW}(X_2\# \ell \CPb).
\]
\end{corollary}

\begin{remark}\label{rem:BlowUpB2+1} 
Because of the chamber structure involved in the domain of the
invariant, the $b_2^+=1$ case needs a bit more care. Suppose that for
a pair of homeomorphic four-manifolds $X_1, X_2$ with $b_2^+(X_i)=1$
and $b_1(X_i)=0$ we have $\vert M_{\SW}(X_1)-M_{\SW}(X_2)\vert \geq
2$. As the wall-crossing can change the invariant by at most 1,
the blow-up formula under this stronger hypothesis implies
that for any $\ell\in \N$ we have 
\[
M_{\SW}(X_1\# \ell \CPb )\neq M_{\SW}(X_2\# \ell \CPb),
\]
consequently the blown-up manifolds $X_1\# \ell \CPb$ and
$X_2\# \ell \CPb$ are smoothly distinct.
\end{remark}

Next, we discuss how the Seiberg-Witten invariants transform under the
circle sum operation.  In this paper, we only consider the case $
\widetilde{X} \#_{\gamma=\gamma'} (S^1 \x Y)$, where $\gamma \subset
\widetilde{X}$ is null-homotopic, $\gamma'$ is isotopic to $S^1 \x
\{\text{pt}\} \subset S^1 \x Y$, and $Y$ is a rational homology
three-sphere, that is, $b_1(Y)=0$.
Thus, for $\widetilde{X} \#_{\gamma=\gamma'} (S^1 \x Y) = \widetilde{X} \#
S(Y)$, we have $ H^2(S(Y); \Z)
\cong H_1(S(Y); \Z)\cong H_1(Y; \Z)$.
As a special case of the Seiberg-Witten blow-up
formula above, we get the following:

\begin{corollary} \label{cor:circlesum-SW}
  Let ${\widetilde {X}}$ be a closed, oriented smooth four-manifold, and let $Y$ be a
  rational homology three-sphere.  For $\s \in {\rm
    {Spin}}^c({\widetilde {X}})$, let $\s ' \in {\rm {Spin}}^c({\widetilde {X}}\# S(Y))$ be its extension.
  Then
  \[
  \SW_{{\widetilde {X}}\# S(Y)}(\s ')=\pm \SW_{{\widetilde {X}}}(\s ).
  \]
  \end{corollary}

Now, the torus surgeries. By~\cite{morgan-mrowka-szabo}, the Seiberg-Witten function
transforms in a controlled manner under torus surgery. Fix a compact,
oriented, smooth four-manifold $M$ with boundary $\partial M =
T^3$. Let $a, b \in H_1(T^3; \Z)$ be two fixed homology classes given
by $a= [S^1\times \{ \text{pt}\} \times \{ \text{pt}\} ]$ and $b = [\{
  \text{pt} \} \times S^1\times \{ \text{pt}\} ]$ for an
identification $\partial M\to S^1\times S^1\times S^1$. For relatively
prime integers $p, q$, let $M(p,q)$ be the closed four-manifold we get
by gluing $D^2\times T^2$ to $M$, so that $[\partial (D^2)]$ maps to
$pa+qb$.

For $M$ a compact, oriented, smooth four-manifold with 
boundary $\partial M = T^3$, let ${\mathfrak
  {s}} \in   {\rm {Spin}}^c(M)$ such that 
$c_1({\mathfrak {s}})$ restricts trivially to $\partial M$. Let ${\mathfrak {S}}(p,q)$ denote
the set of spin$^c$ structures on $M(p,q)$ whose restriction to $M$
agrees with ${\mathfrak {s}}$, and define $F(p,q)$ as the sum of the
Seiberg-Witten invariants of classes in ${\mathfrak {S}}(p,q)$.
\begin{proposition}[\cite{morgan-mrowka-szabo}] \label{prop:torusSW}
  With the notations as above, we  have
\begin{equation}
  F(p,q)= p\cdot F(1,0)+ q\cdot F(0,1).
\end{equation}
\end{proposition}

\smallskip
\begin{remark} \label{rk:FStorusSW}
Based on the above lemma, the following handy criterion is proved in
\cite[Theorem~1]{fintushel-stern:reverse}: Let $X$ be a closed
oriented smooth four-manifold, $\Lambda \subset X$ be a
null-homologous torus with framing $\eta$, and $\lambda \subset
\Lambda$ be a simple closed curve with push-off $\eta_\Lambda$
null-homologous in $\partial (\nu \Lambda)$.  If the four-manifold $(X,
\Lambda, \eta, \lambda, 0)$ has nontrivial Seiberg-Witten invariant,
then the set $\{ (X, \Lambda, \eta, \lambda, (1/k) \mid k \in \Z^+ \}$
contains infinitely many pairwise non-diffeomorphic four-manifolds.
These manifolds are distinguished by their $M_{\SW}$ invariant.
\end{remark}

Finally, we have the following proposition which we will use to distinguish diffeomorphism types through finite covers:
\begin{proposition} \label{prop:multiplecovers} 
Let $\mathcal{F}=\{X_k \, | \, k \in \Z^+\}$ be an infinite family of
closed oriented smooth four-manifolds with isomorphic rational
cohomology ring and fundamental group. Assume that there exists $n \in
\Z^+$ such that each $X_k$ has an $n$--fold cover $\widetilde{X}_k$
with the property that the maximal Seiberg-Witten function $M_{\SW}$
is unbounded on $\mathcal{\widetilde{F}}=\{\widetilde{X}_k \, | \, k
\in \Z^+\}$.  Then, there is an infinite subfamily $\mathcal{F'}
\subset \mathcal{F}$ such that any pair $X_k, X_{k'} \in \mathcal{F'}$
with $k \neq k'$ are nondiffeomorphic.
\end{proposition}

\begin{proof}
Since the fundamental group of a compact manifold is finitely
generated, it can have only finitely many subgroups of a given
index. So, there are only finitely many index $n$ covers of each
$X_k$. If $X_k$ and $X_{k'}$ are diffeomorphic, there would be a
bijection between the diffeomorphism classes of their (finitely many)
$n$--fold covers. Let $K$ be the maximum value of $M_{\SW}$ taken over
all $n$--fold covers of $X_k$. By the hypotheses, there is a $k' > k$
where the $n$--fold cover $\widetilde{X}_{k'}$ of $X_{k'}$ has
$M_{\SW}(\widetilde{X}_{k'}) > K$; consequently, the maximum value
$K'$ of $M_{\SW}$ taken over all $n$--fold covers of $X_{k'}$
satisfies $K' > K$. 
This shows that $X_k$ and $X_{k'}$ are not diffeomorphic, and since we
can keep selecting indices with these properties, the
proposition follows.
\end{proof}

\begin{remark} 
In this article, in all the instances where we distinguish the
diffeomorphism types of smooth four-manifolds by comparing their
finite covers, this is done through their double covers. Most of the
fundamental groups involved in our examples, such as the cyclic groups
$Z_{4m+2}$ in Theorem~\ref{thma}, the groups $Z_2\times \pi _1(Y)$ for
an integral homology three-sphere $Y$ (special cases in
Section~\ref{thm:fpp1}) or $\Z_m \x D_m$ (with $m$ odd) in
Theorem~\ref{thm:fpp2}, have unique index two subgroups, and thus,
unique double covers to compare. The more general statement of the
proposition will be needed when this is not the case, for instance for
some of the examples we produce in Theorem~\ref{thm:fpp1}.
\end{remark}

\smallskip
\section{Signature zero families} \label{sec:signzero}
In proving Theorem~\ref{thma} we will first focus on the case when the
four-manifold has vanishing signature.

For further shorthand, let $R_n:=R_{n,n}=L_2 \# n(\CP \# \CPb)$. In
this section, we produce infinitely many irreducible copies of $R_n$,
for each $n \geq 8$. We present three constructions. The first one
hands us infinitely many irreducible copies of $R_n$, for each $n \geq
10$. The second will work  only for odd $n$, but extends to the $n=9$
case and provides many symplectic examples. The third will be only
for even $n$, and with slightly more involved fundamental group arguments;
it will, however, allow us to extend our results to the $n=8$ case.

Our main building block is the minimal symplectic four-manifold $(Z_m,
\omega_m)$ which is an exotic copy of $\# (2m+1)\, (\CP \,\#\, \CPb)$,
constructed in \cite[Theorem~9]{baykur-hamada:signaturezero} for every
$m \geq 4$. 
This simply connected four-manifold $Z_m$ is derived from a genus
$g=m+1$ symplectic Lefschetz fibration $\mathcal{Z}_m$ over $\Sigma_2$
via Luttinger surgery along a link of $4(m-1)$ Lagrangian tori
$\mathcal{L} \subset \mathcal{Z}_m$.  The map  $\mathcal{Z}_m\to \Sigma _2$
is a
Lefschetz fibration over the genus-2 surface $\Sigma _2$ obtained in
\cite{baykur-hamada:signaturezero}, with four nodal fibers (one of
which is reducible) and a section $\Sigma$ of self-intersection
zero.
  These symplectic fibrations have signature zero, hence after the
  Luttinger surgeries (resulting in $Z_m$ above), we get
  symplectic four-manifolds with signature zero.  We refer the reader
  to \cite{baykur-hamada:signaturezero} for further details; some more
  details can also be found in $\S$\ref{sec:signzero-2} below.

\subsection{First construction} \label{sec:signzero-1} \ 

 In $(Z_m, \omega_m)$, there are two embedded symplectic surfaces of square zero that are of particular interest to us, namely,
\begin{enumerate}[(i)]
\item a genus--$2$ surface $\Sigma$, for each $m \geq 4$, and 
\item a torus $T$, for each $m \geq 5$,
\end{enumerate}
where the complement of each surface in $Z_m$ is simply connected and
contains a surface of odd self-intersection.\footnote{While $Z_m$ in
\cite{baykur-hamada:signaturezero} is constructed for $m \geq 4$, we
are able to argue the existence of the desired symplectic square-zero
torus $T$ only when $m \geq 5$, i.e. when the fiber genus of the
Lefschetz fibration $\mathcal{Z}_m$ we started with is at least $g=m+2
\geq 6$.}

The surface $\Sigma \subset Z_m$ descends from the section of
self-intersection zero in the Lefschetz fibration $\mathcal{Z}_m$. We
can assume that with respect to the Gompf-Thurston symplectic form
taken on $\mathcal{Z}_m$, the surface $\Sigma$ is symplectic
\cite{baykur-hamada:signaturezero}. The meridian $\mu_\Sigma$ of
$\Sigma \subset Z_m$ lies on the geometrically dual surface $F$, which
descends from the symplectic fiber of $\mathcal{Z}_m$, and
$\mu_\Sigma=\prod_{i=1}^{m+1} [a_i, b_i]$ in $\pi_1(Z_m \setminus \nu
\Sigma)$, where $\{a_i, b_i\}$ are the images of the standard
generators of $F \cong \Sigma_{m+1}$ under the inclusion map. The
fundamental group calculation in
\cite[Proof~of~Theorem~9]{baykur-hamada:signaturezero} then shows that
$\mu_\Sigma=1$ in the complement, since all $\{a_i, b_i\}$ are trivial
in $\pi_1(Z_m \setminus \nu F)$. Furthermore, one of the components of
the reducible fiber in $\mathcal{Z}_m$ is disjoint from $\Sigma$ and
descends to a surface of odd self-intersection in $Z_m \setminus \nu
\Sigma$.

The surface $T \subset Z_m$ descends from $\mathcal{Z}_m$
along with a geometrically dual Lagrangian torus $T' \subset Z_m$. As
argued in \cite[Addendum~11]{baykur-hamada:signaturezero}, this pair
of Lagrangian tori is contained in $Z_m$ for every $m \geq 5$, and the
meridian $\mu_T$ of $T$, which is equal to a commutator $[a,b]$
supported on $T'$, is trivial in $\pi_1(Z_m \setminus \nu T)$. The
reducible fiber components in $\mathcal{Z}_m$ are disjoint  from $T$, and
yield surfaces of odd self-intersection in $Z_m \setminus \nu
T$. Using Gompf's trick \cite[Lemma~1.6]{gompf:fibersum}, we can
deform the symplectic form on $Z_m$ around the Lagrangian $T$ so that
it becomes symplectic with respect to the new form. We continue to
denote the symplectic form on $Z_m$ by $\omega_m$.
Note that the genus--$2$ surface $\Sigma$ above is still symplectic after this local deformation.

In addition, the following more technical condition holds: there exists a
null-homologous, framed (say by $\eta$) torus $L\subset Z_m$ disjoint
from $\Sigma$ and $T$, and a loop $\lambda \subset L$, whose push-off
to $\partial \nu L$ is null-homologous in $Z_m \setminus \nu L$, such
that the four-manifold obtained from $Z_m$ by the torus surgery $(T,
\lambda, 0)$
has non-trivial Seiberg-Witten invariant.  We will justify below how
our manifolds meet these conditions.

Lastly, we explain how the null-homologous torus $L \subset Z_m$ comes
into the picture: Let $\widecheck{L}$ be a component of the link of
tori $\mathcal{L} \subset \mathcal{Z}_m$ one performs Luttinger
surgeries along to derive $Z_m$. Let $(\widecheck{Z}_m,
\widecheck{\omega}_m)$ denote the symplectic four-manifold obtained by
performing the prescribed Luttinger surgeries along all components of
$\mathcal{L}$ but $\widecheck{L}$. The core of the Lagrangian
neighborhood of $\widecheck{L} \subset \widecheck{Z}_m$ descends to a
framed null-homologous torus $L \subset Z_m$, with a loop $\lambda
\subset L$ with the claimed properties \cite{fintushel-stern:reverse},
where $\widecheck{Z}_m=(Z_m, L, \eta, \lambda, 0)$. Since
$\widecheck{Z}_m$ is symplectic and $b_2^+(\widecheck{Z}_m)>1$, it has
non-trivial Seiberg-Witten invariant.

We can now spell out our first construction. Recall that $R_n=
L_2\# n (\CP \# \CPb)$ are the manifolds we seek
exotic structures on. The construction comes in slightly different
form depending on the parity of $n$.
 
\noindent{\underline{\textit{even $n=2m+2 \geq 10$}}}: We apply the
$\Z_2$--equivariant construction of $\S$\ref{sec:fibersum} with inputs
$(Z_m, \omega_m)$ and the genus--$2$ symplectic surface
$\Sigma$. Denote the double of $Z_m$ we get in this construction by
$\widetilde{X}_m$ and the resulting quotient four-manifold by $X_m$.

By Lemmas~\ref{lem:pi1} and~\ref{lem:b2}, we have $\pi_1(X_m)=\Z_2$,
$\chi(X_m)=4m+6$, $\sigma(X_m)=0$, and $X_m$ has an odd intersection
form. So $X_m$ is homeomorphic to $R_{2m+2}$. On the other hand, by
Lemma~\ref{lem:sympsum}, $X_m$ is irreducible (and $\widetilde{X}_m$
is minimal). Thus, $X_m$ is an exotic  copy of $R_{2m+2}$.
(Recall that we assumed $m \geq 4$.)

Finally, we explain how to derive an infinite family of distinct exotic
copies. Let $Z_{m,k}=(Z_m, L, \eta, \lambda, \frac{1}{k})$, for each $k \in
\Z^+$, where $L$ is the aforementioned framed null-homologous 
torus. (These are not Luttinger surgeries.)
Let $\widetilde{X}_{m,k}$ be the result of the $\Z_2$--equivariant
normal connected sum of $\S$\ref{sec:fibersum} with inputs $\Sigma \subset
Z_{m,k}$, where $\widetilde{X}_{m,k}$ is equipped with the free
involution $\tau_k$, and let $X_{m,k}:= \widetilde{X}_{m,k} \, /
\tau_k$, for each $k \in \Z^+$ be the output.

It is easily seen that the arguments for $\pi_1(Z_{m,k})=1$ given in
\cite{baykur-hamada:signaturezero} apply verbatim to show that
$\pi_1(Z_{m,k} \setminus \nu \Sigma)=1$, implying that 
$\pi_1(\widetilde{X}_{m,k})=1$. Since surgeries along null-homologous
tori do not change the Euler characteristic, signature, or the spin
type, we conclude that $\widetilde{X}_{m,k}$ and $\widetilde{X}_m$
have isomorphic, odd intersection forms. In turn, $X_{m,k}$ is
homeomorphic to $R_{2m+2}$ for every $k \in \Z^+$.

Through the inclusions of the two copies of $Z_m \setminus \nu
\Sigma$, we get a pair of null-homologous tori $L_i$ in
$\widetilde{X}_m$ and curves $\lambda_i \subset L_i$, for $i=1,2$,
where $\tau(L_1, \lambda_1)=(L_2, \lambda_2)$. The torus surgeries
$(L_1, \lambda_1, 0)$ and $(L_2, \lambda_2, 0)$ in $\widetilde{X}_m$,
say performed in this order, result in symplectic four-manifolds,
which, respectively, are the symplectic normal connected sums of
$(\widecheck{Z}_m, \widecheck{\omega}_m)$ and $(Z_m, -\omega_m)$, and
that of $(\widecheck{Z}_m, \widecheck{\omega}_m)$ and
$(\widecheck{Z}_m, -\widecheck{\omega}_m)$, along the copies of
$\Sigma$ and $\overline{\Sigma}$ via the boundary identification
$\Phi$. That is, each \mbox{$0$--surgery} results in a four-manifold
with non-trivial Seiberg-Witten invariant. Then, applying the main
argument for Remark~\ref{rk:FStorusSW}
twice, we
conclude that infinitely many four-manifolds among $\{
\widetilde{X}_{m,k} \, | \, k \in \Z^+\}$ have pairwise different
Seiberg-Witten invariants. Therefore, after passing to an infinite
index subfamily and relabeling the indices with $k \in \Z^+$ again, we
can assume that $\widetilde{X}_{m,k}$ and $\widetilde{X}_{m,k'}$ have distinct Seiberg-Witten invariants whenever $k \neq k'$.

Hence, $\{ X_{m,k} \, | \, k \in \Z^+\}$ constitutes an infinite family of pairwise non-diffeomorphic 
four-manifolds in the homeomorphism class of $R_{2m+2}$, for each $m \geq 4$. 

\smallskip
\noindent{\underline{\textit{odd $n=2m+1 \geq 11$}}}: 
This time, we apply the $\Z_2$--equivariant construction in $\S$\ref{sec:fibersum} with inputs $(Z_m, \omega_m)$ and $T$. Denote the double by $\widehat{X}_m$ and the quotient by $X'_m$. 

Using Lemmas~\ref{lem:pi1},~\ref{lem:b2} and~\ref{lem:sympsum} as before, we conclude that $X'_m$ is an irreducible four-manifold with $\pi_1(X'_m)=\Z_2$, $\chi(X'_m)=4m+4$, $\sigma(X_m)=0$, and odd intersection form. Therefore, $X'_m$ is an exotic irreducible copy of $R_{2m+1}$.

We derive the infinite family of irreducible copies from the
generalized normal connected sum $\widehat{X}_{m,k}$ of two copies of $Z_{m,k}$
along $T$ and $\overline{T}$ via $\Phi$, which is equipped with a free
involution $\tau_k$ defined in the same fashion as before. Let
$X'_{m,k}:= \widehat{X}_{m,k} \, / \tau_k$, for each $k \in \Z^+$. As
above, we get an infinite family of pairwise non-diffeomorphic
four-manifolds $\{ X'_{m,k} \, | \, k \in \Z^+\}$ in the
homeomorphism class of $R_{2m+1}$, for each $m \geq 5$.

\smallskip
\subsection{Second construction} \label{sec:signzero-2} \  

We present another construction of infinitely many exotic copies of
$R_n$, this time for only \emph{odd} $n \geq 9$. These will \emph{not}
come from $\Z_2$--equivariant constructions; instead, we will perform
slightly different Luttinger surgeries in $\mathcal{Z}_m$ than the
ones that yield the simply connected $Z_m$. (As in the first
construction, we will have $n=2m+1$, $m \geq 4$.) Besides the small
extension to $n=9$, an extra property is that each family of exotic
four-manifolds with $\pi_1=\Z_2$ we get, contains minimal symplectic
ones.

To derive our examples, we modify the construction of $Z_m$ in
\cite{baykur-hamada:signaturezero} by performing one of the Luttinger surgeries differently. Nonetheless, we will need to explain the construction of $Z_m$ some more to justify our claims here, which will also set the ground for the third construction below. We refer the reader to the proofs
of Proposition~2(a) and Theorem~9 in
\cite{baykur-hamada:signaturezero} for some of the essential details,
including the notation we adopt. 

Recall that $Z_m$ is derived from a
genus $m+1$ symplectic Lefschetz fibration $\mathcal{Z}_m$ over
$\Sigma_2$ via Luttinger surgery along a link of $4(m-1)$ Lagrangian
tori $\mathcal{L} \subset \mathcal{Z}_m$. There is a decomposition
\begin{equation} \label{eq:decomp-before}
\mathcal{Z}_m= (\mathcal{K} \,  \cup \,  (\Sigma_{m-4}^2 \x \Sigma_1^1))  \, \cup 
(\Sigma_1^1 \x \Sigma_2) \, \cup \,  (\Sigma_{m}^1 \x \Sigma_1^1) .
\end{equation}
Here the fibration on $\mathcal{Z}_m$ restricts to a non-trivial
Lefschetz fibration $\mathcal{K}$ over $\Sigma_1^1$
\cite{baykur-hamada:signaturezero}, whereas on $\mathcal{Z}_m
\setminus \, \mathcal{K}$, it is simply the projection onto the second
factor for each product of surfaces in the
decomposition~\eqref{eq:decomp-before}.  We label the three pieces of
this decomposition as $\mathcal{A}, \mathcal{B}_0$ and
$\mathcal{C}_0$, in the same order they appear above. See
Figure~\ref{fig:schematic} for a schematic presentation.  Four
components of $\mathcal{L}$
are contained in $\mathcal{B}_0$ and the
remaining $4(m-2)$ in $\mathcal{C}_0$. After the Luttinger surgeries,
the decomposition above yields
\begin{equation} \label{eq:decomp-after}
Z_m= \mathcal{A} \cup \mathcal{B} \cup \mathcal{C}.
\end{equation}

\begin{figure}[htbp]
	\centering
	\includegraphics[width=180pt]{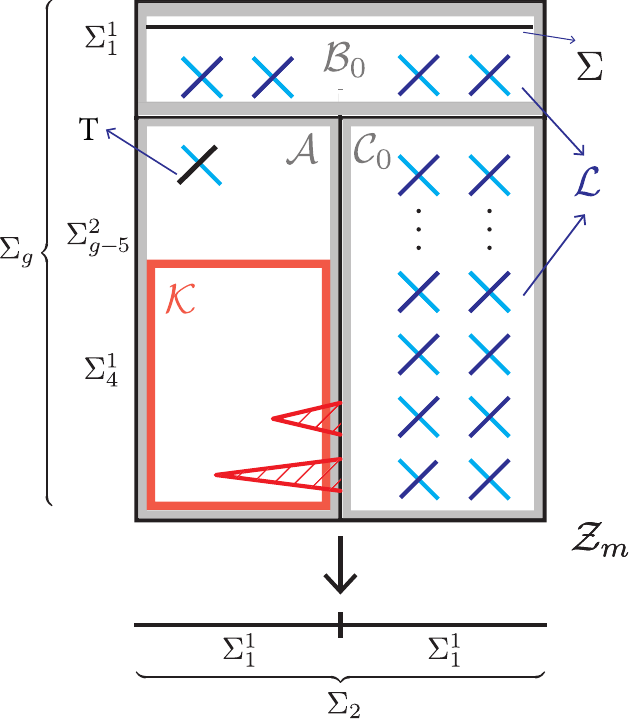}
	\caption{\small The decomposition $\mathcal{Z}_m = \mathcal{A} \cup \mathcal{B}_0 \cup \mathcal{C}_0$. The genus--$2$ surface $\Sigma$ and the torus $T$ are drawn in thick black, the link $\mathcal{L}_m$ in dark blue; all the dual tori are in light blue. The Lefschetz fibration $\mathcal{K}$ is in red.}
	\label{fig:schematic}
\end{figure}

Let $\{a_i, b_i\}$ and $\{x_j, y_j\}$ be the generators of $\pi_1(F)$
and $\pi_1(\Sigma)$, where $F \cong \Sigma_{m+1}$ comes from the
regular fiber of $\mathcal{Z}_m$ and $\Sigma \cong \Sigma_2$ from the
$0$--section. Through the inclusion of $F \cup \Sigma$, these generate
$\pi_1(\mathcal{Z}_m)$.

Let $L_1$ be the component of $\mathcal{L}$ contained in the second
piece $\mathcal{B}_0 \subset \mathcal{Z}_m$, the Luttinger surgery
along which induces the relation $\mu \, a_{m+1} =1$, where
the meridian $\mu$ of
$L_1$ here is a commutator of $x_2$ and $b_{m+1}$. We use the
Lagrangian framing and the same surgery curve (homotopic to $a_{m+1}$
in the complement) as before, but switch the surgery coefficient for the
Luttinger surgery along $L_1 \subset \mathcal{L}$ from $1$ to an
arbitrary positive integer $q$. We perform the same Luttinger
surgeries that were used to get $Z_m$ along the remaining components
of $\mathcal{L}$.

Denote the symplectic four-manifold obtained by the new Luttinger surgery along $\mathcal{L}$ in $\mathcal{Z}_m$ by $Z_m(q)$, and the new surgered piece by $\mathcal{B}(q)$, that is,\[ Z_m(q) = \mathcal{A} \cup \mathcal{B}(q) \cup \mathcal{C}. \]
Replacing $\mathcal{B}$ with $\mathcal{B}(q)$, we now get the relation $\mu
\, a_{m+1}^q =1$ instead. As in the construction of the simply connected $Z_m$, the remaining relators kill all the other generators, so $\pi_1({Z}_m(q))$ is generated by $a_{m+1}$. In particular, $\mu=1$, so we get $a_{m+1}^q=1$. Note that there are no other non-trivial relators involving $a_{m+1}$ in this abelian group. We conclude
that $\pi_1({Z}_m(q)) = \langle a_{m+1} \, | \, a_{m+1}^q \rangle=
\Z_q$.

As we did earlier, we can find a null-homologous torus $L$ in ${Z}_m(p)$,
surgeries along which yields an infinite family of pairwise
non-diffeomorphic 
four-manifolds $\{ Z_{m,k}(q) \, | \, k
\in \Z^+\}$. 

We can now conclude the construction by setting $p=2$.  Setting
$X''_{m,k}$ as the output of the $\Z_2$--equivariant construction of
$\S$\ref{sec:fibersum} with inputs $(Z_{m,k}(2), \Sigma)$, we obtain
infinitely many exotic copies of $R_{2m+1}$, for each $m \geq
4$. (Note that we only get odd $n=2m+1$, as the symplectic
four-manifolds $Z_{m,k}(2)$ with $b_1=0$ necessarily have odd
$b_2^+$.)  When $m \geq 5$, an embedded symplectic torus $T$ with the
same properties as before is also contained in $Z_m(2)$. In this case,
instead of (non-Luttinger) torus surgeries, one can perform knot
surgery \cite{fintushel-stern:knotsurgery} along $T$ using 
fibered knots of different genera
to produce
infinitely many minimal \emph{symplectic} copies of $R_{2m+1}$;
cf. \cite[Addendum~11]{baykur-hamada:signaturezero}.

\smallskip
\subsection{Third construction} \label{sec:signzero-3} \

We present yet another construction, using a codimension zero submanifold contained in $Z_m$ as the main building block. We get infinitely many exotic copies of $R_n$, this time for  \emph{even} $n \geq 8$. The further extension to the $n=8$ case is why we include this variation.

This time we begin with $\mathcal{Z}'_m:=\mathcal{Z}_{m-1}$, the genus $m$ Lefschetz fibration over $\Sigma_2$, and consider the decomposition
\begin{equation} \label{decomposition2}
\mathcal{Z}'_m= (\mathcal{K} \,  \cup \,  (\Sigma_{m-4}^2 \x \Sigma_1^1))  \, \cup 
(D^2 \x \Sigma_2) \, \cup \,  (\Sigma_{m}^1 \x \Sigma_1^1).
\end{equation}
Here ${\mathcal{K}}$ is the same non-trivial Lefschetz fibration over $\Sigma_1^1$ with fibers $\Sigma_4^1$. The main difference from the decomposition~\eqref{eq:decomp-before} we
had for $\mathcal{Z}_m$ is that now the second piece replacing
$\mathcal{B}_0 = \Sigma_1^1 \x \Sigma_2$ is the neighborhood of a section $\Sigma$ of zero
self-intersection, i.e. $\mathcal{Z}'_m=\mathcal{A} \cup \nu \Sigma
\cup \mathcal{C}_0$. 

Let $\mathcal{L}'$ be the link of $4(m-2)$ Lagrangian tori in $\mathcal{Z}'_m$, which are now all contained in $\mathcal{C}_0$. Performing the same Luttinger surgeries along these tori in $\mathcal{C}_0$ yield again $\mathcal{C}$, and the resulting symplectic four-manifold $Z'_m$ decomposes as $Z'_m=\mathcal{A} \cup \nu \Sigma \cup \mathcal{C}$. 

The Luttinger surgeries performed in $\mathcal{C}_0$ are given by the
local model in the proof of \cite[Proposition
  2(b)]{baykur-hamada:signaturezero}. As in the proof of
\cite[Theorem~9]{baykur-hamada:signaturezero}, let $\{a_i, b_i\}$ and
$\{x_j, y_j\}$ be the generators of $\pi_1(F)$ and $\pi_1(\Sigma)$,
where $F \cong \Sigma_m$ descends from the regular fiber of
$\mathcal{Z}'_m$ and $\Sigma \cong \Sigma_2$ from the
$0$--section. Through the inclusion of $F \cup \Sigma$, these generate
$\pi_1(\mathcal{Z}'_m)$, and in turn, $\pi_1(Z'_m)$. After the
Luttinger surgeries along the components of the link $\mathcal{L}'$,
the generators $a_1, b_1, \ldots, a_m, b_m, x_2, y_2$ become normally
generated by two curves $a, b$ on $F$, which get killed in
$\pi_1(Z'_m)$ by the vanishing cycles of $\mathcal{K}$, just like in
the proof of \cite[Theorem~9]{baykur-hamada:signaturezero}. The
surface relator $[x_1,y_1] [x_2,y_2] = 1$ becomes $[x_1,y_1]=1$, so
$\pi_1(Z'_m)$, which is a quotient of $\langle x_1, y_1 \, | \, [x_1,
  y_1] \rangle= \Z^2$, is abelian. No other relators induce any
non-trival relations involving $x_1$ or $y_1$.

The closed four-manifold we have here is the irreducible symplectic four-manifold $(Z'_m, \omega'_m)$ in the homeomorphism class of $(S^2 \x T^2) \# {2m} (\CP \, \# \CPb)$, constructed for each $m \geq 4$ in \cite[Addendum~12]{baykur-hamada:signaturezero}.

We run the $\Z_2$--equivariant construction of $\S$\ref{sec:fibersum}
with inputs $(Z'_m, \omega'_m)$ and $\Sigma$; let $(\widetilde{X'}_m,
\widetilde{\omega}'_m)$ denote the double and $X'''_m$ the quotient
four-manifolds. By Lemma~\ref{lem:pi1} we can arrange the gluing in the
construction so that $\pi_1(\widetilde{X'}_m)=1$, and in turn,
$\pi_1(X''')=\Z_2$. Using Lemmas~\ref{lem:b2} and~\ref{lem:sympsum},
we then conclude that $X'''_m$ is an exotic copy of
$R_{2m}$.

To generate an infinite family, we note that one of the Lagrangian
tori in the link $\mathcal{L}' \subset \mathcal{Z}'_m$ descends to a
null-homologous torus in $Z'_m$, which plays the role of $L$ with a
surgery curve $\lambda \subset L$ we had earlier. Let $(L_i,
\lambda_i)$ be the images of $(L, \lambda)$ under the inclusions of
the two copies of $Z'_m \setminus \, \nu \Sigma$ into
$\widetilde{X'_m}$, so we have $\tau(L_1, \lambda_1)=(L_2,
\lambda_2)$. Let $\widetilde{X'_{m,k}}$ denote the four-manifold
obtained from $\widetilde{X'_m}$ by a pair of equivariant torus
surgeries $(L_1, \lambda_1, k)$ and $(L_2, \lambda_2, k)$. Then
$\widetilde{X'_{m,k}}$ can be equipped with an involution $\tau_k$, so
we take $X'''_{m,k}:=\widetilde{X'_{m,k}} / \, \tau_k$. Following the
same arguments as earlier, we conclude using Remark~\ref{rk:FStorusSW}
that we get infinitely many exotic copies of $R_{2m}$, for
every $m \geq 4$.

\bigskip
We have proved cumulatively:

\begin{thm} \label{thm:signzero}
There are infinitely many, pairwise non-diffeomorphic smooth four-manifolds in the homeomorphism class of $R_{a,b}$, for all $a, b \in \N^2$ with $a=b \geq 8$. \linebreak In each family, any pair of four-manifolds either have different Seiberg-Witten invariants or their universal double covers do.
\end{thm}

\medskip

\smallskip
\section{Signature negative one families} \label{sec:signone}

Our aim in this section is to prove:

\begin{thm} \label{thm:exoticSignatureminusone}
There are infinitely many, pairwise non-diffeomorphic, smooth
four-manifolds in the homeomorphism class of $R_{n,n+1}$, for every $n
\in \N2$ with $n \geq 0$. In each family, for odd $n$, any pair of
four-manifolds have different Seiberg-Witten invariants, and for even
$n$, their universal double covers do.
\end{thm}

For a shorthand notation, let us set $S_n:=R_{n,n+1}= L_2\# n \CP \#
(n+1)\CPb$.  Our examples of families of exotic smooth structures on
$S_n$ will be obtained differently for even and odd $n$. For $n$ even,
the $\Z_2$--equivariant construction we are going to carry out will be
more involved than the constructions in the odd $n$ case. The case
$n=0$ here was already covered in~\cite{stipsicz-szabo:definite}, and 
examples for odd $n$ have been present in the literature
\cite{ABBKP, akhmedov-park, torres}; instead of replicating
similar arguments, we will refer to those manifolds to complete the
proof of the theorem.

\subsection{The equivariant construction of simply connected double covers}
\label{ssec:construction} \

In the following, fix $n=2m$, for some $m \in \N$.  Consider the
four-torus $T^4=T^2\times T^2$ together with the fibration map
$T^2\times T^2\to T^2$ given by projection to the second factor. Fix a
fiber $F=T^2\times \{ \text{pt} \}$ and a section $\sigma =\{ s\}
\times T^2$ of this trivial torus fibration.  With the identification
$T^2=[0,1]^2/\sim$, we choose $s=(\epsilon , \epsilon)$ for
sufficiently small $\epsilon >0$.  As explained in
\cite[$\S$4.1]{stipsicz-szabo:definite} (cf.  ~\cite{akhmedov-park,
  smith} for the origins of this idea), by replacing one
$S^1$-component of the fiber with a 2-braid, the homology class $2[F]$
can be represented by a torus, which we call ${\mathcal{T}}$.
This torus can be constructed by taking a suitable double cover of the fiber
within its trivial neighborhood.
In explicit terms, $\mathcal{T}$ is given by the
points $(z_1,z_2,z_3,z_4)$ of $T^4$ and $\theta\in \{ \pi/4, 5\pi/4
\}$ which satisfy
\begin{equation}\label{eq:torus}
(z_3,z_4) =(2\epsilon +\sqrt{2}\epsilon \cdot {\rm
  cos}(g(z_1)+\theta), 2\epsilon+ \sqrt{2}\epsilon\cdot {\rm
  sin}(g(z_1)+\theta)).
\end{equation}

The torus ${\mathcal{T}}$ intersects the section $\sigma$ in two
points, and we assume that the chosen fiber $F$ is disjoint from
${\mathcal{T}}$. Note that if we equip $T^4$ with the product
symplectic structure (coming from volume forms on the $T^2$-factors)
then both the chosen torus $F$ and the chosen section are symplectic,
and the braided torus ${\mathcal{T}}$ is also chosen so that it is a
symplectic submanifold.

In constructing the examples, we will use several building blocks, the
main one $V(p_1/q_1, p_2/q_2)$ is described in
\cite[$\S$3.1]{stipsicz-szabo:definite}. For the sake of completeness,
we recall its definition here. Viewing $T^4$ as the quotient of
$[0,1]^4$, we have the four circles $x=(t,0,0,0)$, $y = (0,t,0,0)$,
$a=(0,0,t,0)$, $b = (0,0,0,t)$, with $t\in [0,1]$; these loops are
based at $x_0= (0,0,0,0)$ and provide a free abelian generating set of
$\pi _1 (T^4, x_0)\cong \Z ^4$.  Furthermore, for the fixed
real numbers {$c_1,c_2,c_2',c_3,c_4', c_4 $} in $[1/2,3/4]$ with
$c_2'<c_2$ and $c_4'<c_4 $ (and coordinates $(z_1,z_2,z_3,z_4)$ on
$T^4$) we consider the 2-dimensional disjoint tori $T_1\subset T^4$
given by the equations $\{ z_2=c_2, z_4 = c_4\} $, $T_2$ given by $\{
z_1= c_1, z_4= c_4'\}$, and
$T_3$ given $\{ z_2=c_2', z_3=c_3\}$.

Consider the disk $D\subset T^2=[0,1]^2/\sim$ centered at $(2\epsilon,
2\epsilon)\in T^2$ with radius $\epsilon$, and notice that $\epsilon$
can (and will) be chosen so small that the above loops as well as the
three tori $T_1,T_2$ and $T_3$ are in $(T^2\setminus D)\times
(T^2\setminus D)$. Using the symplectic structure on $T^4$ as above,
the tori $T_1, T_2, {T_3}$ are Lagrangian, and applying torus surgery
with coefficient $p_1/q_1$ on $T_1$ (with simple closed curve
$\gamma _1\subset T_1$ homologous to $x$) and with coefficient $p_2/q_2$
along $T_2$ (with simple closed curve $\gamma _2\subset T_2$ homologous to
$a$) we get the four-manifold $V(p_1/q_1, p_2/q_2)$. Let $M_0$ denote
$(D\times T^2)\cup (T^2\times D)$ which can be viewed in $T^4$ and in
$V(p_1/q_1, p_2/q_2)$ as well. We denote the complement of the
interior of $M_0$ in $V(p_1/q_1, p_2/q_2)$ by $V_0(p_1/q_1, p_2/q_2)$.

Recall that $n=2m$ is fixed. Consider $m$ four-manifolds
$V^i(u_i/r_i)$ (with $i=1,\ldots , m$ and $u_i, r_i\in \Z$) we get
from $T^4$ by performing $u_i/r_i$ surgery along the Lagrangian torus
$T_1$ with its Lagrangian framing and with simple closed curve
homologous to $a$.
Further $m$ four-manifolds $V_j(U_j/R_j)$ are fixed
which we get from $T^4$ (again, with $j=1,\ldots , m$ and $U_j,R_j \in \Z$)
by performing torus surgery along the torus $T_3\subset T^4$ defined above,
equipped with its  
Lagrangian framing, with surgery curve $b$, 
and surgery coefficient $U_j/R_j$.

We take the normal connected sum of these $2m$ four-manifolds with our
original $V(p_1/q_1,p_2/q_2)$ as follows.  Recall that the fiber $F$
is fixed in $T^4$, and since all torus surgeries were performed
disjoint from a neighborhood of $F$, a copy of this torus (with
a slight abuse of notation also denoted by $F$) appears in all the
four-manifolds $V(p_1/q_1,p_2/q_2), V^i(u_i/r_i)$ and $V_j(U_j/R_j)$.

In the sectional torus $\sigma = \{ s\}\times T^2 $ (with its usual identification
with $\{ s\} \times [0,1]^2/\sim$) fix $2m$ points
$\sigma _i$, for $i=1,\ldots , m$, and 
$\tau _j$, for $j=1,\ldots , m$:
at $(\frac{1}{2}, \frac{i}{10m})$ for $\sigma _i$
and $(\frac{1}{2}, \frac{2}{10}+\frac{j}{10m})$ for $\tau _{m-j}$;
see Figure~\ref{fig:basepoints}.
Fix small disjoint disks centered at these points.
\begin{figure}[htb]
\begin{center}
\setlength{\unitlength}{1mm}
 \includegraphics[height=6cm]{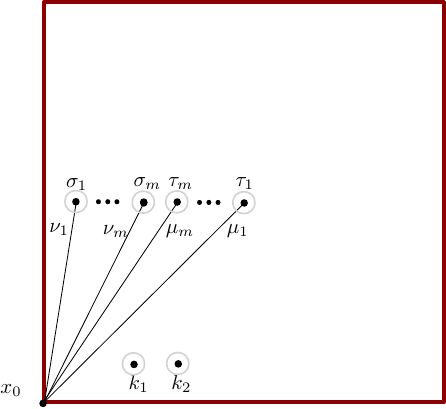}
\end{center}
\caption{\quad The choices of $\sigma _i $ and $\tau _j$.}
\label{fig:basepoints}
\end{figure}

Now take the normal connected sum of $V(p_1/q_1, p_2/q_2)$ with $V^i(u_i/r_i)$ by
identifying the boundary of the small tubular neighbourhood of
$F\subset V^i(u_i/r_i)$ with the boundary of a sufficiently small
tubular neighbourhood of $T^2\times \{ \sigma_i\}$ (given by
$T^2$-times the small disk fixed above) using the diffeomorphism
lifted from the obvious identification of $F\subset V^i(u_i/r_i)$ with
$T^2\times \{ \sigma _i\}$. Take further $m$ normal connected sums in a similar
manner with $V_j(U_j/R_j)$, using now the tori $T^2\times \{ \tau
_j\}$.  Denoting the vector $(u_1/r_1, \ldots , u_m/r_m)$ by ${\bf
  {u}}/{\bf {r}}$ (and $(U_1/R_1, \ldots , U_m/R_m)$ by ${\bf
  {U}}/{\bf {R}}$), the above construction provides the four-manifold
\[
Y=Y(p_1/q_1, p_2/q_2, {\bf {u}}/{\bf {r}}, {\bf {U}}/{\bf {R}}).
\]
In the construction of $Y$ we can glue the sections $\sigma =\{
s\}\times T^2$ (having a copy in each one of the above four-manifolds
$V(p_1/q_1,p_2/q_2)$, $V^i(u_i/r_i)$ and $V_j(U_j/R_j)$) together to
get an embedded surface $\Sigma \subset Y$ with genus $g=2m+1$ and
self-intersection zero.

Note that the torus ${\mathcal {T}}$ of Equation~\eqref{eq:torus} can
be viewed in $V(p_1/q_1, p_2/q_2)$ and indeed in $Y$, giving rise to a
torus (also denoted by ${\mathcal {T}}$) intersecting $\Sigma$ in two
points $k_1,k_2$.  Consider the (singular) surface $\Sigma \cup
{\mathcal {T}}$, resolve the positive transverse intersection at $k_1$
and blow up the four-manifold $Y$ at $k_2$.  By considering the proper
transform of $\Sigma \cup {\mathcal {T}}$, we get an embedded surface
$G_1$ of genus $g+1=2m+2$ and of self-intersection zero in
\begin{equation}\label{eq:TheW}
W(p_1/q_1, p_2/q_2, {\bf {u}}/{\bf {r}}, {\bf {U}}/{\bf {R}}):=
Y(p_1/q_1, p_2/q_2, {\bf {u}}/{\bf {r}}, {\bf {U}}/{\bf {R}})\# \cpkk .
\end{equation}
Notice that if $p_1, p_2, u_i, U_i\in \{ \pm 1\}$ then the
corresponding $W$ is a symplectic manifold. In addition, as the fiber,
the braided surface and the section are all symplectic submanifolds of
$T^4$, and all operations (smoothing, blow-up and normal connected
sum) are symplectic, it follows that $G_1$ is a symplectic surface in
$W$.

Consider now two copies of $(W,G_1):=(W(p_1/q_1, p_2/q_2, {\bf
  {u}}/{\bf {r}}, {\bf {U}}/{\bf {R}}), G_1)$ and define $ {\widetilde
  {X}} (p_1/q_1, p_2/q_2, {\bf {u}}/{\bf {r}}, {\bf {U}}/{\bf {R}})$
as their normal connected sum along the two copies of $G_1$, where the
gluing map $\Phi_g$ is chosen as follows. First, define the map 
$S^1\times \Sigma _{g+1}\to S^1\times \Sigma _{g+1}$ as the antipodal
map on the $S^1$-factor and the orientation-reversing involution
$\phi_g$ on $\Sigma _{g+1}= \Sigma_{m+1}^1 \cup \Sigma_{m+1}^1$, which
exchanges the two subsurfaces $\Sigma_{m+1}$ by reflecting
$\Sigma_{g+1}$ along their common boundary.
See Figure~\ref{fig:reflection}.
\begin{figure}[htb]
\begin{center}
\setlength{\unitlength}{1mm}
 \includegraphics[height=7.6cm]{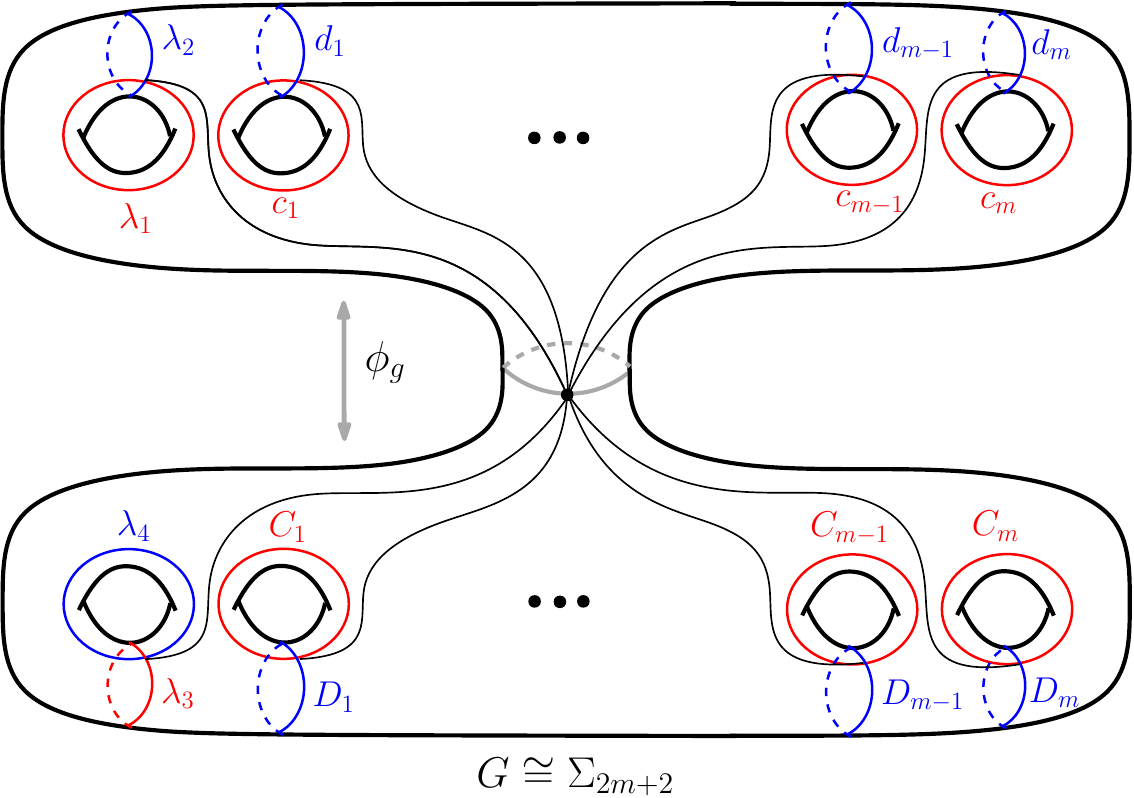}
\end{center}
\caption{\quad The reflection $\phi_g$ of the genus $2m+2$ surface
  $G$; the fixed point set is the gray circle in the middle. The
  fundamental group generators are given by the loops $\{\lambda _i,
  c_j, d_j, C_j, D_j\}$ with the respective arcs connecting them to
  the base point in the middle. The homology generators are $\{ \lambda_i, 
  c_j, d_j, C_j, D_j\}$.}
\label{fig:reflection}
\end{figure}
In order to apply this map, we need to identify $\partial ( W(p_1/q_1,
p_2/q_2, {\bf {u}}/{\bf {r}}, {\bf {U}}/{\bf {R}})\setminus G)$ with
$S^1\times \Sigma _{g+1}$ --- that is, we need to fix a framing on
$G_1$.  {In this step we will use the ideas described in
  \cite[$\S$4.2]{stipsicz-szabo:definite}. Indeed, in $V(p_1/q_1,
  p_2/q_2)$ the genus-2 surface has been equipped with a framing as it
  is detailed in \cite[$\S$4.2]{stipsicz-szabo:definite}.  In the
  further four-manifolds $V^i(u_i/r_i)$ and $V_j(U_j/R_j)$ the
  sections $\{ s\}\times T^2$ come with a natural framing from the
  product structure of $T^4$ before the torus surgeries. These
  framings glue together under the normal connected sums to provide a
  suitable framing for $G_1\subset W(p_1/q_1, p_2/q_2, {\bf {u}}/{\bf
    {r}}, {\bf {U}}/{\bf {R}})$.}  Using this framing, we get an
embedding $D^2\times G_1\subset W(p_1/q_1, p_2/q_2, {\bf {u}}/{\bf
  {r}}, {\bf {U}}/{\bf {R}})$; let $G$ denote $\{ {\rm {pt}}\} \times
G_1$ where ${\rm {pt}}\in \partial D^2$.

As $\Phi _g$ is fixed point free, the four-manifold
${\widetilde {X}} (p_1/q_1, p_2/q_2, {\bf {u}}/{\bf {r}}, {\bf {U}}/{\bf {R}})$
admits a fixed point free, orientation preserving involution $\iota _g$,
by combining
the identification of the two $W$-pieces with the gluing diffeomorphism.

\begin{definition}
We let ${\widetilde {X}}_k$ denote the special case where
all $u_i=U_i=p_2=1$, $r_i=R_i=q_1=q_2=-1$ and $p_1=k\in \N$.
\end{definition}

\subsection{The fundamental group calculation} \ 

This subsection is devoted to the verification the following:
\begin{thm}\label{thm:SimplyConnected}
  For a fixed $m$ (and hence $g=2m+1$)
  the manifolds $X_k$ ($k\in \N$) are simply connected.
\end{thm}

{As a first step, we determine a generating system and some relations
  in the fundamental group of $W_k=W(-k, -1, {\bf {-1}}, {\bf {-1}})$.
  In order to do so, fix the following convention: in the torus $\{
  z_1=z_2=0\}$ (containing the basepoint $x_0$) in $T^4$, and
  containing the points $\sigma _1, \ldots \sigma _m, \tau _m ,\ldots
  , \tau _1$, fix paths $\nu _1, \ldots \nu _m, \mu _m, \ldots , \mu
  _1$, as shown by Figure~\ref{fig:basepoints}.

This torus gives rise to a similar torus after the torus surgeries,
hence we can view this torus, together with the points and paths, also
in $V(-k, -1)$.  When taking the normal connected sum of $V(-k,-1)$ with the
manifolds $V^i(-1)$ and $V_j(-1)$, we connect sum the torus $\{
z_1=z_2=\epsilon\}\subset V(-k,-1)$ with the sections of the $V^i$'s
and the $V_j$'s to construct the genus--$g$ surface, which we use to
construct $G_1$. Using these paths and a pair of transversely
intersecting simple closed curves in each section (which are also tori),
we find the loops $a_i, b_i$ (from $V^i(-1)$) and $A_j,B_j$ (from $V_j(-1)$)
($i,j=1, \ldots , m$) in
$W_k$. Using the framing we push these curves $a_i, b_i, A_j, B_j$ to the
boundary of the tubular neighbourhood of $G_1$ --- these push-offs will be
denoted by the same letters.

\begin{proposition}\label{prop:FundGroupCalcBasic}
  The fundamental group $W_k$ (equipped with the base point
  $x_0\in V(-k, -1)$)
is generated by the loops $x,y, a, b, a_i, b_i$, for $i=1, \ldots ,m$,
and $A_j, B_j$, for $j=1, \ldots , m$.
Furthermore, these generators satisfy the relations
\begin{equation}\label{eq:first}
[x,y]=[a,b]=1
\end{equation}
\begin{equation}\label{eq:second}
[x,a]=[y,a]=[b^{-1},y^{-1}]^kx^{-1}=[b^{-1},x^{-1}]a^{-1}=1
\end{equation}
and 
\begin{equation}\label{eq:third}
[x,a_i]=[y,a_i]=[x,b_i]=[b_i^{-1},y]a_i^{-1}=1,
\end{equation}
\begin{equation}\label{eq:fourth}
[x,A_j]=[y,B_j]=[x,B_j]=[A_j^{-1},y]B_j^{-1}=1.
\end{equation}
The same relations hold in $W_k \setminus \nu (G_1)$ with the
exception of the relations in~\eqref{eq:first}.
\end{proposition}

\begin{proof}
For the proof, we rely on the arguments of
Baldridge-Kirk~\cite{BK}, where the generators and some key relations of the
fundamental group of $V(-k, -1)$ and $V_0(-k,-1)$ (the complement of
$M_0$ in $V(-k, -1)$) are identified.  Indeed, the loops $x,y,a,b$ (as
we specified earlier) generate $\pi _1 (V(-k,-1))$, and similar
elements $x_i, y_i, a_i, b_i$ generate $\pi _1(V^i(-1))$ and $X_j,
Y_j,A_j, B_j$ generate $\pi _1 (V_j(-1))$. By the
Seifert-Van Kampen theorem, in the normal connected sums 
the generators $x_i$ and $X_j$ get all
identified with $x$ and similarly all $y_i$ and $Y_j$ are identified
with $y$. These observations verify that the set of homotopy elements
given in the proposition indeed generate the fundamental group of $W_k$.

The relations in $W_k$ follow from the relations found by \cite{BK}
in $V(-k,-1)$ and in $V_0(-k,-1)$. The same goes for the relations in
$V^i$'s and $V_j$'s and in the complements of their fibers and
sections.

The meridian of $G_1\subset W_k$ may constitute an additional
generator when considering $\pi _1 (W_k\setminus G_1)$, but this
element is equal to $[x,y]$.
As $G_1$ intersects the horizontal and vertical tori of $T^4$ (and
hence of $V(-k,-1)$) in a point each, in $\pi _1 (W_k \setminus G_1)$
the two relations in~\eqref{eq:first} seize to hold.
\end{proof}

Let the simple closed curves in $T^2\times \{ \text{pt} \}\subset T^4$
intersecting each other in a single point generating its first
homology denoted by $\lambda _1, \lambda _2$,
while in $\{ \text{pt} \} \times
T^2\subset T^4$ a similar pair of curves by $\lambda _3, \lambda _4$,
as shown in \cite[Figure~3]{stipsicz-szabo:definite}.
In constructing $G_1$, we normal
connect sum copies of $V^i$, adding genus (with homology generators
$c_i,d_i$, $i=1, \ldots , m$) to the section  
and copies of $V_j$, adding genus (with homology
generators $C_j, D_j$ with $j=1, \ldots ,m$) to the section. In
conclusion, the first homology of the surface $G_1$ is generated by
these curves, see Figure~\ref{fig:reflection} for a pictorial
presentation of them.  Connecting the intersection points of these
pairs to the base point (as shown in Figure~\ref{fig:reflection}),
these simple closed curves naturally provide elements in the
fundamental group of the surface $G_1$ (with base point in the middle),
which we denote by the same letters. Using the framing, we push these
curves (and the base point) to $G$; in this process
$\lambda _1$ gives rise to $u_1$, $\lambda _2$ to $v_1$, $\lambda _3$ to
$u_2$ and $\lambda _4$ to $v_2$, cf. \cite[Figure~4]{stipsicz-szabo:definite}.
The push-offs of $c_i, d_i, C_j, D_j$ can be chosen so
that these curves are homotopic to $a_i,b_i, A_j, B_j$
fixed earlier ($i,j=1, \ldots , m$).

For determining the fundamental group of ${\widetilde {X}}_k$, we will
need to understand the map $F\colon \pi _1 (G, x_0)\to \pi _1 (W_k,
x_0)$ induced by the embedding $i\colon G \to W_k$. (Here the base
point $x_0$ is in $G=\{ {\rm {pt}}\} \times G_1\subset W_k$.)
The map $F$ induced by the
embedding of \cite[Lemma~4.4]{stipsicz-szabo:definite} generalizes to
this setting, and the adaptation of the proof of that lemma shows:
\begin{lemma}
The above choices of curves can be made so that 
$F$ maps
\begin{itemize}
\item $c_i$ to $a_i$ and $C_j$ to $A_j$ $(i,j=1, \ldots , m)$
\item $d_i$ to $b_i$ and $D_j$ to $B_j$ $(i,j=1, \ldots , m)$
\item $v_2$  to $a$ and $u_2$ to $b$, and finally
\item $v_1$ to $x$ and $u_1$ to $y^2\cdot\Pi _{i=1}^m[a_i,b_i]\cdot
\Pi_{j=m}^1[A_j,B_j]$. 
\end{itemize}
\end{lemma}
\begin{proof}
The homotopies claimed in the first two bullet points are rather
obvious in the manifolds $V^i$ and $V_j$. The last two statements
follow as it is detailed in \cite[Lemma~4.4]{stipsicz-szabo:definite},
with the simple modification that now we deal with the normal circles
of all normal connected sums, which can be expressed as commutators of
the chosen generators $a_i,b_i, A_j,B_j$; their order in the product
is determined by the choice of the paths given in
Figure~\ref{fig:basepoints}.
\end{proof}

The action of the reflection $\phi _g$ on $\pi _1(G, x_0)$ is easy to
determine.  Note first that $\phi _g$ moves the base point $x_0$ to
$\phi _g(x_0)$ with equal coordinates on $G_1$, but antipodal on the
$S^1$-factor.  The two points $x_0$ and $\phi _g (x_0)$ can be
connected by a semi-circle in one of the fibers of the fibration
$S^1\times G_1\to G_1$ (given by the chosen framing on $G_1$).  This
choice identifies $\pi _1 (G, x_0)$ with $\pi _1 (\phi _g (G), \phi _g
(x_0))$, and using this identification together with our convention
fixed in Figure~\ref{fig:reflection}, the homomorphism $(\phi _g)_*$
maps
\begin{itemize}
\item $u_1$ to $v_2$ (and $v_2$ to $u_1$);
\item $u_2$ to $v_1$, (and $v_1$ to $u_2$);
\item it interchanges $c_i$ and $C_i$ ($i=1,\ldots ,m$), and
\item interchanges $d_i$ and $D_i$ ($i=1,\ldots ,m$).
\end{itemize}

Now we are ready to conclude the computation of the fundamental group
of ${\widetilde {X}}_k$:

\begin{proof}[Proof of Theorem~\ref{thm:SimplyConnected}]
In the construction of the four-manifold ${\widetilde {X}}_k$ we take
the normal connected sum of two copies of $W_k$ along the surface
$G_1\subset W_k$. The gluing map $\phi _g$ patching $W_k\setminus \nu
(G_1)$ with the other copy is also fixed.

As $b$ and $y$ from one copy correspond to $x$ and $a$ in the other,
the relations $[b^{-1},y^{-1}]x^{-1}$ on one side and $[x,a]$ on the
other imply that $x$ in the first copy of $W_k$ is trivial in the
fundamental group of ${\widetilde {X}}_k$. This then shows that (as
$[b^{-1}, x^{-1}]a^{-1}$ is a relation) $a$ is also trivial; as this
argument works for both sides, $x$ and $a$ are trivial on both
sides. As these are identified via the gluing map with $b$ and $y$,
those elements are also trivial. The relations $[b_i^{-1},y]a_i^{-1}$
imply the triviality of $a_i$ ($i=1, \ldots ,m$) and similarly the
relations $[A_j ^{-1} y]B_j ^{-1}$ imply the triviality of $B_j$. As
these elements are identified with $A_i$ and $b_j$, we see that all
generators are trivial in $\pi _1({\widetilde {X}}_k)$, concluding the
argument.
\end{proof}

\subsection{The topology of ${\widetilde {X}}_k$} \

We first verify that

\begin{proposition}
${\widetilde {X}}_k$ is homeomorphic to
$(2g-1)\CP \# (2g+1)\CPb$.
\end{proposition}
\begin{proof}
As a torus surgery does not change the signature $\sigma$ and the
Euler characteristic $\chi$, these numerical invariants of
${\widetilde {X}}_k$ are $\sigma ({\widetilde {X}}_k)=-2$ and $\chi
({\widetilde {X}}_k)=4g+2$.  As the smooth four-manifold ${\widetilde
  {X}}_k$ is simply connected, and its signature is not divisible by
$16$, by Rokhlin's theorem ${\widetilde {X}}_k$ has odd
intersection form.  The result follows from Freedman's classification
theorem \cite{freedman} for smoothable topological manifolds.
\end{proof}

Finally, we claim that
\begin{proposition}
For fixed $n=2m$ 
the manifolds ${\widetilde {X}}_k$ and ${\widetilde {X}}_{k'}$ are
diffeomorphic if and only if $k=k'$.
\end{proposition}

This will follow from the computation of their Seiberg-Witten invariants:

\begin{lemma}
The four-manifold $X_m$ has exactly two basic classes $\pm K$, and the Seiberg-Witten
values on these classes are $\SW_{X_m}(\pm K)=\pm m^2$.
\end{lemma}
\begin{proof}
  For $g=1$ this claim was verified in \cite{stipsicz-szabo:definite}
  for a slightly different gluing diffeomorphism, but the arguments of
  \cite[Theorem~4.7]{stipsicz-szabo:definite} apply to the current
  setting verbatim. Indeed, we can use the manifolds constructed in
  \cite{stipsicz-szabo:definite} in proving the $g=1$ case of the lemma.
  
Assuming $g=2m+1>1$, note first that the manifold ${\widetilde {X}}_1$
is constructed from a symplectic four-manifold using Luttinger
surgeries, hence itself is symplectic. This shows that (as
$b_2^+({\widetilde {X}}_k)=2g-1>1$) the cohomology classes $\pm
c_1({\widetilde {X}}_1, \omega)$ are basic classes, with
Seiberg-Witten values $\pm 1$.  Using the adjunction inequality it can
be shown that these are the only basic classes.  Indeed, in
$W_1=W(-1,-1, {\bf {-1}}, {\bf {-1}})$ there are $2m$ pairs of tori,
intersecting each other within the pairs transversally once (otherwise
disjoint), and having self-intersection 0.  Gluing two copies of $W_1$
(along $G$) together, we see $8m$ tori of self-intersection zero, on
which (by the adjunction inequality of
Proposition~\ref{prop:adjunction}) all basic classes evaluate
trivially. We have that $b_2({\widetilde {X}}_k)=4g=8m+4$, and for the
evaluation of a basic class on the remaining four homology elements
the argument of \cite[Lemma~4.7]{stipsicz-szabo:definite} applies.
Indeed, the remaining four homology elements can be represented by
embedded surfaces as follows: the surface $G_1$ (a self-intersection
zero, genus-$(g+1)$ surface) is intersected transversely once by a
self-intersection zero, genus--2 surface we get by gluing together two
fibers from $V(-1,-1)$; these two homology classes are denoted by $x$
and $y$ respectively.  The blow-up on either side of the construction
provides a $(-1)$--sphere which intersects $G_1$ twice. Completing
these twice-punctured spheres with punctured tori on the other side,
we get two disjoint genus-2 surfaces with self-intersection $-1$ and
intersecting $G_1$ twice. Let $q_1, q_2$ denote their homology
classes; so $q_1\cdot q_2=0$, $q_i\cdot y=0$ and $q_i\cdot x=2$ for
$i=1,2$.

As a basic class $K$ for a symplectic four-manifold (being of simple
type) has $K^2=3\sigma +2\chi$, which for ${\widetilde {X}}_1$ is
equal to $8g-2$, we get that (as $K$ is characteristic) it should
evaluate $\pm 2g$ on $x$, $\pm 2$ on $y$, and the two evaluations need
to have the same sign. Once again, to have $K^2=8g-2$, we need that
$K$ evaluates on $q_1, q_2$ as $\pm 1$. Assume that $K(x)=2g$ (and so
$K(y)=2$).  The exact same argument as in
\cite[Lemma~4.7]{stipsicz-szabo:definite} implies that
$K(q_1)=K(q_2)=1$, implying that $K=\pm c_1({\widetilde {X}}_1)$.

Now Proposition~\ref{prop:torusSW} concludes the argument for the
proof of the lemma. Indeed, the formula shows that ${\widetilde
  {X}}_k$ also has exactly two basic classes, provided we determine
the Seiberg-Witten function for the manifold $W(0,-1, {\bf {-1}}, {\bf
  {-1}})\cup W(-k,-1, {\bf {-1}}, {\bf {-1}})$ we get by performing
0-surgery on one of the surgery tori. After performing 0-surgery on
the torus $T_1$, we can represent the fiber with an embedded sphere,
which in $W(0,-1, {\bf {-1}}, {\bf {-1}})\cup W(-k,-1, {\bf {-1}},
{\bf {-1}})$ connects to a torus, representing a homology class with a
torus on which potential basic classes evaluate as 2 (in order to have
the required square), providing the desired contradiction.

Then Proposition~\ref{prop:torusSW}, applied twice, determines
the value of the Seiberg-Witten invariant on the unique (pair) of basic
classes $\pm K$, giving
\[
\SW_{{\widetilde {X}}_k}(\pm K)=\pm k^2,
\]
verifying that the positive integer $k$ is a diffeomorphism invariant of
$X_k$, concluding the argument.
\end{proof}

\subsection{Exotic four-manifolds with signature $-1$ and fundamental group $\Z _2$} \

We are ready to complete the main result of this section.

\begin{proof}[Proof of Theorem~\ref{thm:exoticSignatureminusone}]
  We are going to present infinitely many exotic copies of the
  four-manifold $S_n=L_2\# n \CP \# (n+1)\CPb$, with pairwise
  (virtually) distinct Seiberg-Witten invariants.  By 'virtual
  Seiberg-Witten invariants' we mean the invariants of the manifold at
  hand, together with the Seiberg-Witten invariants of its double
  cover.

For odd
$n$, the desired examples can be obtained from exotic simply connected
families via Luttinger surgeries, as we did in
$\S$\ref{sec:signzero-2}. For $n=1$, the desired example is
constructed in \cite[$\S$9]{akhmedov-park}; for $n > 1$, the
construction of such examples is outlined in
\cite[Remark~11]{akhmedov-park}. The corresponding infinite families
of four-manifolds have pairwise distinct Seiberg-Witten invariants.

For even $n$, by choosing $n=2m$, our construction in this section
yields an infinite family of closed, simply connected four-manifolds
$\{ {\widetilde {X}}_k\}$ with free, orientation-preserving
involutions $\tau_k$, which pairwise have distinct Seiberg-Witten
invariants. The quotient four-manifolds $\{X_k={\widetilde {X}}_k /
\tau_k \, | \, k \in \Z^+\}$ provide the desired family of exotic
four-manifolds in the homeomorphism class of $S_n$. Note that here we
distinguished the exotic structures through the Seiberg-Witten
invariants of their universal double covers. (Indeed, all these
four-manifolds have vanishing Seiberg-Witten invariants, as
$b_2^+(S_n)=n$ is even, which implies that the moduli spaces of
Seiberg-Witten solutions are odd dimensional.)
\end{proof}

\smallskip
\section{Proof of Theorem~\ref{thma}} \label{sec:proofA}

We are now ready to prove our first main theorem.

\begin{proof}[Proof of Theorem~\ref{thma}]
We begin with the $\pi_1=\Z_2$ case. We are going to generalize our results to the $\pi_1=\Z_{4m+2}$ case, for any further $m \in \N$ afterwards.

Let $\{X_{n,k} \}$ be an infinite family of pairwise non-diffeomorphic
four-manifolds in the homeomorphism class of $R_{n,n}=L_2 \# n \, \CP
\# n \, \CPb$ we constructed in Theorem~\ref{thm:signzero} for every
$n \geq 8$, where $k$ runs through values in $\Z^+$. For $k \neq k'$,
either $X_{n,k}$ and $X_{n,k'}$ have different Seiberg-Witten
invariants, or their universal double covers $\widetilde{X}_{n,k}$ and
$\widetilde{X}_{n,k'}$ do.  By Corollary~\ref{cor:BlowUpDistinction}
it follows that the blow-ups $X_{n,k} \# \ell \, \CPb$ and $X_{n,k'}
\# \ell \, \CPb$ are pairwise non-diffeomorphic, because either they
have different Seiberg-Witten invariants, or their universal double
covers $\{\widetilde{X}_{n,k} \# {2\ell} \, \CPb \}$ and
$\{\widetilde{X}_{n,k'} \# {2\ell} \, \CPb \}$ do.  Therefore,
$\{X_{n,k} \# \ell \, \CPb \}$ constitutes an infinite family of
pairwise non-diffeomorphic four-manifolds homeomorphic to $R_{n,
  n+\ell}$.  Reversing the orientations of these four-manifolds, we
get a similar family homeomorphic to $R_{n+\ell , n}$.  We conclude
that there are infinitely many exotic copies of $R_{a,b}$ for every
$a, b \geq 8$.

Next, let $\{X'_{n,k} \}$ be an infinite family of pairwise
non-diffeomorphic four-manifolds in the homeomorphism class of
$R_{n,n+1}= L_2 \# n \, \CP \# \, (n+1) \CPb$ we constructed in
Theorem~\ref{thm:exoticSignatureminusone} for every $n \geq 0$, and $k
\in \Z^+$. Once again, for $k \neq k'$, the four-manifolds $X_{n,k}$
and $X_{n,k'}$ have different virtual Seiberg-Witten invariants,
so their blow-ups are also pairwise non-diffeomorphic
(cf. Corollary~\ref{cor:BlowUpDistinction}). It follows that
$\{X'_{n,k} \# \ell \, \CPb \}$ constitutes an infinite family of
pairwise non-diffeomorphic four-manifolds in the homeomorphism class
of $R_{n, n+\ell +1}$. We get a similar family in the homeomorphism
class of $R_{n+\ell +1, n}$ by reversing orientations.

Therefore, there are infinitely many exotic copies of $R_{a,b}$ also when $a, b \leq 7$, provided $a \neq b$, which realize the remaining homeomorphism types promised in the statement of the theorem when $\pi_1=\Z_2$.

To generalize to any $\pi_1=\Z_{4m+2}$, we do the following: 

When the family $\{X_{n,k}\}$ (or $\{X'_{n,k}\}$) is constructed as
the quotient of simply connected four-manifolds
$\{\widetilde{X}_{n,k}\}$, we apply the $\Z_2$--equivariant circle sum
construction of $\S$\ref{sec:circlesum} with inputs
$Z=\widetilde{X}_{n,k}$ and the lens space $Y=L(2m+1, 1)$. Denote each quotient
four-manifold by $X_{n,k}(m)$. By Lemma~\ref{lem:circlesum-topology},
we have $\pi_1(X_{n,k}(m))=\Z_{2m+1} \x \Z_2= \Z_{4m+2}$. By
Corollary~\ref{cor:circlesum-SW}, the family $\{X_{n,k}(m)\}$ consists of
four-manifolds with distinct virtual Seiberg-Witten invariants.

When the family $\{X_{n,k}\}$ (or $\{X'_{n,k}\}$) is derived by a
certain Luttinger surgery with surgery coefficient $q=2$ from a simply
connected symplectic four-manifold (for each $k \in \Z^+$), we can
instead perform a $q=4m+2$ Luttinger surgery along the same torus
(with the same surgery curve) to derive a four-manifold we denote by
$X_{n,k}(m)$. As noted for instance in ~$\S$\ref{sec:signzero-2}, we
get $\pi_1(X_{n,k}(m))=\Z_{4m+2}$ and $\{X_{n,k}(m)\}$ also consists
of four-manifolds with distinct Seiberg-Witten invariants.

These two modified constructions then provide infinitely many exotic
copies of $R_{a,b}(m)$ with signatures $\sigma=0$ (when $a=b \geq 8$)
and $-1$ (when $a=b-1\geq 0$). As before, we obtain infinitely many exotic
copies of all the other $R_{a,b}(m)$ claimed in the theorem
by blow-ups and by reversing
orientations.
\end{proof}

\smallskip
\begin{remark} \label{rk:remaining}
A word about the exceptions in Theorem~\ref{thma} for $\pi _1=\Z _2$:
For $n\in \{1, 2, 3 \}$ the double cover $(2n+1)(\cpk \# \cpkk )$ of
$R_n=L_2 \# n\, (\CP \# \CPb)$ is not known to admit exotic copies;
whereas for each $n \in \{4,5,6,7\}$ the main theorem of
\cite{baykur-hamada:signaturezero} provides infinitely many exotic
copies, but we did not manage to find an orientation preserving free
involution on these four-manifolds.
\end{remark}

\begin{remark} \label{rk:literature}
Exotic smooth structures on some of these topological four-manifolds
were given in earlier works. For example, for odd $b_2^+$, symplectic
examples with cyclic fundamental groups (and in particular with
$\pi_1=\Z_{4m+2}$), can be derived by performing Luttinger surgeries
from the examples of \cite{ABBKP}. There are also many classical
examples from complex algebraic geometry, such as elliptic surfaces
$E(n)_{p,q}$ with gcd$(p,q)=4m+2$. The most extensive study of
examples with odd $b_2^+$ and finite cyclic fundamental group was
carried out by Torres in \cite{torres}. Definite examples and examples
with even $b_2^+$ and $b_2^-$, which are also derived from minimal
simply connected four-manifolds equipped with free involutions, are
more recent; see \cite{beke-koltai-zampa,  harris2024exotic,
LLP,
  stipsicz-szabo:definite, stipsicz-szabo:bigdefinite}.  One
can also use Bauer's non-vanishing result for stable cohomotopy
Seiberg-Witten invariants \cite{bauer} to obtain some exotic families with
cyclic $\pi_1$ and even $b_2^+$ and $b_2^-$ from this collection of
examples and the simply connected ones.  
\end{remark}

\begin{remark}  \label{rk:irred}
We expect that, for $\frac{1}{5}(a-4) \leq b\leq 5 a+4$, an analogous
statement to Theorem~\ref{thma} with infinite families of
\emph{irreducible} smooth structures can be obtained using the methods
of this paper. The additional inequalities imposed on the
four-manifold $X$ homeomorphic to $R_{a,b}$ translates to having
$c_1^2=2 \chi + 3 \sigma \geq 0$ for at least one orientation of $X$
and its double cover, which makes it possible for us to detect
irreducible smooth structures on them via Seiberg-Witten theory. Our
methods can also be applied to obtain partial results for
four-manifolds with \emph{even} intersection forms, but without access
to blow-ups (which would yield odd intersection forms) to expand the
geography, the possible results are largely restricted to the realm of
irreducible four-manifolds. Unresolved mysteries surrounding the
topology of smooth spin four-manifolds, such as the outstanding
$11/8$--conjecture, make it currently impossible to approximate any
exhaustive statements as in Theorem~\ref{thma}.
\end{remark}

\smallskip
\section{Non-complex fake projective planes} \label{sec:fpp}

In this final section, we prove Theorem~\ref{thmb}.  We are going to
obtain our irreducible fake projective planes in two ways, one
leveraging the equivariant circle sum construction, and the other via
equivariant torus surgeries. 
We first record the following fact:

\begin{proposition} \label{prop:FPPirreducible}
  A complex fake projective plane, and more generally, a
  compact complex surface with Chern numbers $c_1^2=3 c_2 >0$ is smoothly irreducible.
\end{proposition}

\begin{proof}
  Any compact complex surface $X$ with $c_1^2(X) > 0$ is K\"{a}hler, and
when $c_1^2(X)=3 c_2(X)$, by Yau's celebrated result \cite{Yau},
it is a quotient of the complex ball. This implies that $X$ is aspherical.

As shown by Kahn \cite{kahn}, a locally flatly embedded three--sphere
$Y$ in any compact four--manifold $X$ is null-homotopic if and only if
it bounds a contractible four--manifold in $X$.
Since $\pi_3(X)=0$, we 
conclude that for any smooth decomposition $X=X_1 \# X_2$, one
of the $X_i$ is a homotopy four--sphere. This implies that
$X$ is irreducible.

Alternatively,  by the Bogomolov-Miyaoka-Yau inequality, the complex surface $X$ is
necessarily minimal, and by well-known applications of gauge theory, a
minimal K\"{a}hler surface $X$ with $b_2^+(X)>1$ is smoothly
irreducible. In the $b_2^+(X)=1$ case, namely when $X$ is a complex
fake projective plane, we recall that $\pi_1(X)$ is a sublattice in
$PU(2,1)$.  By a classical result of Mal'cev \cite{Malcev},
every finitely generated linear group is residually
finite, consequently, $\pi_1(X)$ is residually finite. Let $X' \to X$
be a covering corresponding to a proper finite index subgroup of
$\pi_1(X) \neq 1$. The finite cover $X'$ is a K\"{a}hler surface with
$c_1^2(X')=3c_2(X')$ and $b_2^+(X')>1$, so, as argued above, it is
smoothly irreducible, in turn implying irreducibility of $X$.
\end{proof}

\subsection{First construction} \label{sec:fpp-1} \

Let $Z$ be one of the minimal smooth four-manifold homeomorphic to
$\CP \# 3 \CPb$, equipped with an orientation-preserving free
involution $\tau$ and with non-trivial Seiberg-Witten invariants
constructed in \cite{stipsicz-szabo:definite} or in
$\S$\ref{sec:signone}.  Let $Y$ be a rational homology three-sphere.

We run the $\Z_2$--equivariant circle sum construction of
$\S$\ref{sec:circlesum} with inputs $(Z, \tau)$ and
$Y$. By our assumption on $Y$, the spun $S(Y)$ is a rational homology
four-sphere. It follows that if $Z$
is a rational homology $\CP \#
3\CPb$, then so is $\widetilde{X} = Z \# S(Y)$. Consequently, by
Lemma~\ref{lem:circlesum-topology} the quotient four-manifold $X$ is a
rational homology $\CPb$ with $\pi_1(X)=\pi_1(Y) \x \Z_2$.

Note that the four-manifold $Z$ is irreducible (as it has two basic classes), while the double cover $\widetilde{X}$ is clearly reducible.
We claim that

\begin{proposition}
The four-manifold $X$ is 
  irreducible.
  \end{proposition}

\begin{proof}
We first note
that $\pi_1(X)$ cannot be expressed as a free product of non-trivial
groups: indeed, $\pi _1 (X)$ has non-trivial center, while the free
product of non-trivial groups has trivial center.
Assume now that $X$ is reducible, so $X=X_1 \# X_2$, where neither $X_i$
is a homotopy four-sphere. By the above observation and the
Seifert-Van Kampen theorem, one of $X_i$, say $X_1$, should be simply
connected. So, by the Mayer-Vietoris long exact sequence, $X_1$ is a
homotopy $\CPb$ and $X_2$ is a rational homology four-sphere with
$\pi_1(X_2)=\pi_1(X)$. This implies that the double cover decomposes
as $\widetilde{X}= \widetilde{X}_2 \# \, 2 X_1$, where $\widetilde{X}_2$
is a double cover of $X_2$.

Consider now the image of the set $\{ \s \in {\rm {Spin}}^c(\widetilde{X})
\mid \SW_{\widetilde{X}}(\s )\neq 0 \}$ under the map $c_1\colon {\rm
  {Spin}}^c(\widetilde{X}) \to H^2 (\widetilde{X}; \Q )$ we get by
associating the first Chern class with rational coefficients to a
spin$^c$ structure.  The formulae of Proposition~\ref{prop:blowup-SW}
and Corollary~\ref{cor:circlesum-SW} show that the decomposition
$\widetilde{X}=Z \# S(Y)$ gives that the image has two elements
(as $Z$ has two basic classes), while
the decomposition $\widetilde{X}= \widetilde{X}_2 \# \, 2 X_1$ implies
that the image has at least four elements, providing the desired
contradiction, verifying the irreducibility of $X$.
\end{proof}

\begin{thm}\label{thm:fpp1} 
Let $R$ be the fundamental group of a rational homology three-sphere.
Then, there is an infinite family $\mathcal{F}(R)=\{F_k(R) \, | \, k
\in \Z^+\}$ of pairwise non-diffeomorphic, irreducible fake projective
planes with $\pi_1(F_k(R))=R \x \Z_2$. None of the fake projective
planes $F_k(R)$ admit symplectic or complex structures.
\end{thm}

\begin{proof}
Let $\{ Z_k \, | \, k \in \Z^+\}$ be an infinite family of irreducible
homotopy $\CP \# 3 \CPb$s, where each $Z_k$ is equipped with an
orientation-preserving free involution $\tau_k$ and $\SW_{Z_k} \nequiv
\SW_{Z_{k'}}$ when $k \neq k'$.  Such a family is given in
$\S$\ref{sec:signone}.

Running the above line of arguments for each $Z_k$ and a fixed
rational homology three-sphere $Y$, we get an infinite family $\{ X_k
\, | \, k \in \Z^+\}$ of irreducible
rational homology $\CPb$s. By Proposition~\ref{prop:multiplecovers}, passing to an infinite subindex family if needed, we can moreover assume that this family consists of pairwise non-diffeomorphic four-manifolds. 
Letting $F_k(R):=\overline{X}_k$, the four-manifold
$X_k$ with reversed orientation, we get the desired family
$\mathcal{F}(R)$, where $R= \pi_1(Y)$.

As the universal cover of $F_k(R)$ is not contractible, by Yau's
theorem \cite{Yau} none of these fake projective plane may admit a
complex structure.  However, to argue that $F_k(R)$ cannot admit a
symplectic structure, we rely on the topology of the particular
$Z_k$ we employed in our construction. As shown below, an alternate
proof of the non-existence of a complex structure on $F_k(R)$ will
also follow from this observation.

We observe that there is an embedded torus of square $-2$ in 
$Z_k \setminus \, \nu(\gamma_k)$. Here $\gamma_k$ is the
$\tau_k$--equivariant curve we use in the circle sum. One can realize
this torus by patching --- during the normal connected sum --- a copy of the
exceptional sphere in $W_k$ that intersects $G_1$ at two points
with another copy coming from the second $W_k$ piece; see
$\S$\ref{sec:signone}. This torus is also contained in $Z_k \# S(Y)$,
which, with opposite orientation, is the double cover of
$F_k(R)$. Because these oppositely oriented double covers contain
embedded tori of square $+2$, they violate the Seiberg-Witten
adjunction inequality (as their small perturbation Seiberg-Witten
invariants are non-trivial). Consequently, these
four-manifolds cannot admit
symplectic structures. Since admitting a symplectic structure is a
property preserved under coverings, we conclude that no $F_k(R)$ can
admit a symplectic structure. Lastly, because $b_1(F_k(R))$ is even,
it admits a complex structure if and only if it admits a K\"{a}hler
structure, providing another argument for $F_k(R)$ admitting no
complex structure.
\end{proof}

\begin{example} \label{ex:H1ofFPPs}
  Suppose that $A$ is a finite abelian group
  with primary decomposition $A\cong \Z_{n_1}\oplus \cdots \oplus \Z_{n_\ell}$ into cyclic groups.
  Then the group $G=\Z _2\oplus A$ is the first homology group of
  infinitely many smoothly distinct fake projective planes by applying the
  above circle sum construction by choosing $Y$ to be the connected
  sum of the lens spaces $L(n_i, 1)$ for $i=1, \ldots , \ell$.
  \end{example}

\subsection{Second construction} \label{sec:fpp-2} \

Our next family of irreducible fake projective planes have  more
restricted first homology than the examples above. Nonetheless, they
provide further examples with
additional fundamental groups.

We modify the construction of the minimal symplectic homotopy $\CP \#
3 \CPb$ in $\S$\ref{sec:signone}
(cf. \cite[$\S$4]{stipsicz-szabo:definite}) via equivariant Luttinger
and non-Luttinger surgeries along Lagrangian tori. Here, the
$\Z_2$--equivariance refers to performing simultaneous surgeries along
pairs of tori mapped to each other under the involution, with matching
choices.

Our main building block is $T^4 \# \CPb$, taken with the same embedded
\mbox{genus--$2$} symplectic surface $G_1$ of square zero and the
same pair of Lagrangian tori as in the building block
$W(p_1/q_1,p_2/q_2)$. (Recall that this four-manifold was introduced as the
blow-up of the building block $V(p_1/q_1, p_2/q_2)$.) However, we will
take different surgery coefficients for the pair of torus surgeries
this time. We consider two families: $W(-1/m, -1/n)$ and $W(-\ell ,
-1/m)$, where $m,n, \ell \in \Z^+$. Here, we use the notations
introduced in $\S$\ref{ssec:construction}.

Applying the normal connected sum construction of $\S$\ref{sec:fibersum}, we let
$X(m,n)$ denote the double of $W(-1/m, -1/n)$. (In the notation of $\S$\ref{ssec:construction}, this manifold is constructed by taking $g=1$, hence
working in $T^4$ and its blow-up $T^4\# \CPb$.)
This manifold is equipped with the
orientation-preserving free involution $\tau_{m,n}$, and we let
$F(m,n)$ denote the quotient four-manifold $X(m,n) \, / \,
\tau_{m,n}$, taken with the \emph{opposite} orientation.  Similarly,
let $X_m(\ell)$ denote the double of $W(-\ell, -1/m)$ with involution
$\tau'_{\ell ,m}$, and let $F_m(\ell )$ be the quotient four-manifold $X_m(\ell)
\, / \, \tau'_{\ell ,m}$, also taken with the opposite orientation.

By the irreducibility of $X(m,n)$ and $X_m(\ell)$
every $F(m,n)$ and $F_m(\ell )$ is irreducible.  From the homology
relations induced by the Luttinger surgeries in $T^4 \# \CPb$, we see
that $H_1(X(m,n); \Q)=H_1(X_m(\ell ); \Q)=0$. If we denote any of these
quotient four-manifolds
$F(m,n), F_m(\ell )$
by $F$ and its double cover by $q\colon
\widetilde{F} \to F$, then from the short exact sequence
\[
1 \longrightarrow \pi_1(\widetilde{F}) {\longrightarrow} \pi_1(F) \longrightarrow \Z_2 \longrightarrow 1 
\]
of their
fundamental groups we derive (for instance, by
applying the abelianization and $\Q$--tensoring functors) that
\begin{cor} 
For the four-manifold $F$ defined above we have $H_1(F;
\Q)= 0$, $\chi(F)=3$ and $\sigma(F)=1$, hence 
$F$ is a fake projective plane. \qed
\end{cor}

Each $X(m,n)$ is symplectic, since all the torus surgeries involved in
their construction are Luttinger surgeries. This is not the case for
$X_m(\ell )$, and in fact, a calculation of their Seiberg-Witten
invariants shows that they do not admit symplectic structures
once $\ell >2$, cf. 
\cite{stipsicz-szabo:definite}. Here, we can invoke Propositions~\ref{prop:torusSW} and~\ref{prop:multiplecovers} to conclude that, possibly after passing to an infinite subindex family,  
any $F_m(\ell )$ and $F_{m}(\ell ')$
are non-diffeomorphic for $\ell \neq \ell '$.

From their construction, we observe that the doubles $X(m,n)$ and
$X_m(\ell )$ also contain embedded tori of square $-2$. As in the
proof of Theorem~\ref{thm:fpp1}, this implies that $F(m,n)$ and
$F_m(\ell )$ (covered by oppositely oriented $X(m,n)$ and
$X_m(\ell)$) do not admit any symplectic or complex structures either.

Lastly, we turn to the fundamental groups we get. While our
construction here yields a fairly rich family of fundamental groups
for the fake projective planes $F(m,n)$ and $F_m(\ell )$, we restrict
our calculation to a subfamily, where we can explicitly identify the
resulting groups. Indeed, in the following we will focus exclusively
on $F_m(\ell)$ with $m $ odd.

We first argue that $\pi_1(X_m(\ell))=\Z_m \x \Z_m$.
Indeed, when computing $\pi _1(W_m(\ell))$, we need
to modify  the relators in~\eqref{eq:first} and \eqref{eq:second} as follows:
replace the
relator
\[
[b^{-1}x^{-1}]a^{-1}\qquad {\rm{ with}} \qquad [b^{-1}x^{-1}]a^{-m}
\]
and the relator
\[
  [b^{-1}, y^{-1}]^mx^{-1} \qquad {\rm { with}} \qquad [b^{-1}, y^{-1}]^\ell x^{-1}.
  \]
  
Then, the same arguments as in the proof of Proposition~\ref{prop:FundGroupCalcBasic}
  show that the group $\pi_1(X_m(\ell))$ is a quotient of
  $\Z_m \x \Z_m$.  
  The first homology group of this manifold
  is easier to determine from the relations induced by the torus surgeries, giving
  $H_1(X_m(\ell); \Z)=\Z_m \x \Z_m$, veryfing
  $\pi _1 (X_m(\ell))\cong \Z _m \x \Z _m$.

\begin{lemma}
For any $\ell  \in \Z^+$ and odd $m \in \Z^+$,
we have $\pi_1(F_m(\ell))=\Z_m \x D_m$,  where $D_m$
is the dihedral group of order $2m$.
\end{lemma}

\begin{proof}
  To simplify the notation, let us set $\widetilde{X}=X_m(\ell)$,
  $X=F_m(\ell)$, and $m=2n+1$ for some $n \in \N$.  Let $\tau$ denote
  the free involution on $\widetilde{X}$ and $q\colon \widetilde{X}
  \to X$ the corresponding double covering. The induced short exact
  sequence
\begin{equation} \label{eq:coveringpi1}
1 \longrightarrow \pi_1(\widetilde{X}) \overset{{q}_*}{\longrightarrow} \pi_1(X) \longrightarrow \Z_2 \longrightarrow 1 
\end{equation}
splits on the right by the homomorphism $s\colon \Z_2 \to \pi_1(X)$ defined by $s(1)=\gamma$, where $\gamma=q(\alpha)$ for $\alpha$ a path lifting $\gamma$ and such that $\alpha \cup \tau(\alpha)$ is a loop doubly covering $\gamma$. Thus, $\pi_1(X)$ is a semi-direct product $\pi_1(\widetilde{X})
\rtimes \Z_2 = (\Z_m \x \Z_m ) \rtimes \Z_2$.

To calculate the action of $\Z_2$ on $\Z_m \x \Z_m =
q_*(\pi_1(\widetilde{X})) \triangleleft \pi_1(X)$, let the generators
$(1,0)$ and $(0,1)$ correspond to loops
$\beta_1=q(\widetilde{\beta}_1)$ and $\beta_2=q(\widetilde{\beta}_2)$,
where each $\widetilde{\beta}_i$ is a loop in a separate copy of
$\partial W(-\ell,-1/m) \setminus \nu \Sigma$ in $W_m(\ell)$. It
follows that $\gamma \beta_1 \gamma^{-1} =q (\alpha
\widetilde{\beta}_1 \,\tau(\alpha))= q(\widetilde{\beta}_2)= \beta_2$,
and similarly, we get $\gamma \beta_2 \gamma^{-1}=\beta_1$. We deduce
that the $\Z_2$ action on $\Z_m \x \Z_m$ is given by $(r,s) \mapsto
(s,r)$ for all $r, s \in \Z_m$.

Denoting the generators of $\Z_m \x \Z_m$ by $a$ and $b$ and that of
$\Z_2$ by $c$, we get the following presentation for $\pi_1(X)$:
\begin{align*}
   &  &\langle a, b,c \ |& \ a^m, b^m, c^2, [a,b], cac^{-1}b^{-1} \rangle  \\
    =&  &\langle a, b,c, x, y \ |& \ a^m, b^m, c^2, [a,b], cac^{-1}b^{-1}, x^{-1}ab, y^{-1}a^{-1}b \rangle  \\
    =&  &\langle a, b,c, x, y \ |& \ a^m, b^m, c^2, [a,b], cac^{-1}b^{-1}, x^{-1}ab, y^{-1}a^{-1}b, \\
    & & & \  x^m, y^m, [x,y], cxc^{-1}x^{-1}, cyc^{-1}y, a^{-2} x y^{-1}, b^{-2} xy
    \rangle  \\
     =&  &\langle c, x, y \ |& \ x^m, y^m, c^2, [x,y], [x,c], cyc^{-1}y \rangle ,
\end{align*}
where we relied on the relations $a=(a^2)^{n+1}$ and $b=(b^2)^{n+1}$ in $\pi_1(X)$, made possible by $m=2n+1$ being odd. So we get 
\[ 
\pi_1(X)=\langle x \ | x^m \rangle \x \langle c,y \ | \ y^m, c^2, cyc^{-1}y \rangle = \Z_m \x D_m
\]
as claimed. 
\end{proof}

We have therefore proved:

\begin{thm}\label{thm:fpp2}
  For every fixed odd $m\in \Z ^+$, there exists an infinite family
  \linebreak $\mathcal{F}(m)= \{F_\ell (m) \, | \, \ell \in \Z^+\}$ of pairwise
  non-diffeomorphic, irreducible fake projective planes with
  $\pi_1(F_\ell (m))=\Z_m \x D_m$. None of the fake projective planes
  $F_\ell (m)$ admit symplectic or complex structures.
\end{thm}

\begin{remark}
  Theorem~\ref{thmb} is a combined statement of Theorems~\ref{thm:fpp1}
and~\ref{thm:fpp2}. Yet, many more non-complex irreducible fake
projective planes with different fundamental groups than the ones we
singled out can be constructed by combining the two constructions we
have presented here.
\end{remark}

\begin{remark}
As we noted in Example~\ref{ex:H1ofFPPs}, we get irreducible fake
projective planes with first homology $A \x \Z_2$, for $A$ any desired
finite abelian group $A$. It follows that exactly 41 out of 50 complex
fake projective planes have the same integral homology as one of the
non-symplectic fake projective planes we have produced here;
cf. \cite[Appendix]{GalkinKatzarkovMellit}.
\end{remark}

\begin{remark} 
We can refine our infinite families $\mathcal{F}(R)=\{F_k(R) \, | \, k
\in \Z^+\}$ and $\mathcal{F}(m)= \{F_\ell (m) \, | \, \ell \in \Z^+\}$
so that each one contains infinitely many, pairwise non-diffeomorphic
irreducible fake projective planes in the same homeomorphism class.

For the fake projective planes constructed in
$\S$\ref{sec:fpp-1}, the key observation is the following:
Because the quotients $Z_k/ \tau_k$ and $Z_{k'}/ \tau_{k'}$ are
homeomorphic, there is an equivariant homeomorphism between $(Z_k,
\tau_k)$ and $(Z_{k'}, \tau_{k'})$. Performing the equivariant circle
sum with $S^1 \x Y$ along circles that match under this homeomorphism,
we can extend this equivariant homeomorphism between $(Z_k
\#_{\gamma=\gamma'} S^1 \x Y, \tau_k')$ and
$(Z_{k'}\#_{\gamma=\gamma'} S^1 \x Y, \tau_{k'}')$, where $\tau'_k$,
$\tau'_{k'}$ are the free involutions on the circle
sums. Hence, we can assume that the quotients $F_k(R)$ and
$F_{k'}(R)$ are homeomorphic.

For the families of fake projective planes discussed in
$\S$\ref{sec:fpp-2}, we note that the result of Hambleton and Kreck
in \cite{hambleton-kreck2} implies that there are finitely many
homeomorphism types with a fixed finite fundamental group and integral
intersection form. Since the fake projective planes $F_\ell(m)$ have
finite fundamental groups, the existence of an infinite exotic family
is a consequence of the pigeonhole principle.
\end{remark}

\bigskip

\bibliography{exotic_pi1Z2.bib}
\bibliographystyle{plain}

\end{document}